\documentclass[pdflatex,sn-mathphys-num]{sn-jnl}% Math and Physical Sciences Numbered Reference Style
\usepackage{graphicx}%
\usepackage{multirow}%
\usepackage{amsmath,amssymb,amsfonts}%
\usepackage{amsthm}%
\usepackage{mathrsfs}%
\usepackage[title]{appendix}%
\usepackage{xcolor}%
\usepackage{textcomp}%
\usepackage{manyfoot}%
\usepackage{booktabs}%
\usepackage{algorithm}%
\usepackage[noEnd=true,commentColor=black]{algpseudocodex}%
\usepackage{listings}%

\usepackage{tikz}
\usepackage{comment}

\usepackage{longtable, booktabs}
\usepackage{array}

\definecolor{myblue}{RGB}{0, 50, 150}

\DeclareMathOperator*{\argmax}{arg\,max}
\DeclareMathOperator*{\argmin}{arg\,min}

\newcommand{\inprod}[2]{{\left \langle #1,#2 \right \rangle}} % inner product

\newcommand{\st}{\textrm{s.t.}}
\newcommand{\tr}{\textrm{tr}}
\newcommand{\Diag}{\textrm{Diag}}

\newcommand{\rank}{\textrm{rank}}

\theoremstyle{thmstyleone}%
\newtheorem{proposition}{Proposition}% 
\newtheorem{lemma}{Lemma}% 

\theoremstyle{thmstyletwo}%

\theoremstyle{thmstylethree}%
\newtheorem{definition}{Definition}%

\usepackage{background}

\begin{document}

\SetBgContents{Published at \url{https://doi.org/10.1007/s10589-026-00764-6}}      % a watermark
\SetBgPosition{current page.center}
\SetBgAngle{0}                                    % rotate
\SetBgColor{gray}                                 % color
\SetBgScale{1.5}                                  % scale
\SetBgHshift{0}                                   % location x=0 for center
\SetBgVshift{9cm} 

\title[Exact and Heuristic Algorithms for Constrained Biclustering]{Exact and Heuristic Algorithms for Constrained Biclustering}

%%=============================================================%%
%% GivenName	-> \fnm{Joergen W.}
%% Particle	-> \spfx{van der} -> surname prefix
%% FamilyName	-> \sur{Ploeg}
%% Suffix	-> \sfx{IV}
%% \author*[1,2]{\fnm{Joergen W.} \spfx{van der} \sur{Ploeg} 
%%  \sfx{IV}}\email{iauthor@gmail.com}
%%=============================================================%%

\author{\fnm{Antonio M.} \sur{Sudoso}}\email{antoniomaria.sudoso@uniroma1.it}

\affil{\orgdiv{Department of Computer, Control, and Management Engineering ``Antonio Ruberti"}, \orgname{Sapienza University of Rome}, \orgaddress{\street{Via Ariosto 25}, \city{Rome}, \postcode{00185}, \country{Italy}}}

%%==================================%%
%% Sample for unstructured abstract %%
%%==================================%%

%shoud belong or not to the same bicluster

\abstract{Biclustering, also known as co-clustering or two-way clustering, simultaneously partitions the rows and columns of a data matrix to reveal submatrices with coherent patterns. Incorporating background knowledge into clustering to enhance solution quality and interpretability has attracted growing interest in mathematical optimization and machine learning research. Extending this paradigm to biclustering enables prior information to guide the joint grouping of rows and columns. We study constrained biclustering with pairwise constraints, namely must-link and cannot-link constraints, which specify whether objects should belong to the same or different biclusters. As a model problem, we address the constrained version of the k-densest disjoint biclique problem, which aims to identify k disjoint complete bipartite subgraphs (called bicliques) in a weighted complete bipartite graph, maximizing the total density while satisfying pairwise constraints. We propose both exact and heuristic algorithms. The exact approach is a tailored branch-and-cut algorithm based on a low-dimensional semidefinite programming (SDP) relaxation, strengthened with valid inequalities and solved in a cutting-plane fashion. Exploiting integer programming tools, a rounding scheme converts SDP solutions into feasible biclusterings at each node. For large-scale instances, we introduce an efficient heuristic based on the low-rank factorization of the SDP. The resulting nonlinear optimization problem is tackled with an augmented Lagrangian method, where the subproblem is solved by decomposition through a block-coordinate projected gradient algorithm. Extensive experiments on synthetic and real-world datasets show that the exact method significantly outperforms general-purpose solvers, while the heuristic achieves high-quality solutions efficiently on large instances.}


\maketitle

\section{Introduction}\label{sec1}
Biclustering is a fundamental problem in data mining and machine learning, aiming to simultaneously group rows (samples) and columns (features) of a data matrix to identify submatrices, known as biclusters, that exhibit coherent patterns \cite{mirkin1996mathematical, cheng2000biclustering}. Unlike traditional clustering, which partitions either the rows or the columns, biclustering reveals localized, context-dependent relationships that may be invisible in a one-way analysis. This capability is particularly valuable in domains where relationships between samples depend on specific subsets of features. For example, in gene expression analysis, two genes may exhibit co-expression patterns only under a subset of experimental conditions \cite{cheng2000biclustering}; similarly, in document clustering, two documents may be closely related only with respect to a specific subset of terms or topics \cite{dhillon2001co}. For an extensive overview of biclustering applications, see the survey in \cite{busygin2008biclustering} and references therein. While this flexibility has demonstrated practical utility in many applications, it also introduces significant computational challenges. Therefore, striking the right balance between solution quality, interpretability, and efficiency remains a central focus of current research.

Clustering is an unsupervised learning task, as it operates exclusively on unlabeled data to discover hidden patterns. However, when performed without any form of guidance, clustering algorithms may produce groupings that fail to align with domain expertise or the intended analytical goals. To address this limitation and improve both the quality and interpretability of the results, researchers have introduced mechanisms to incorporate prior knowledge into the clustering process \cite{basu2008constrained, davidson2007survey}. A common approach involves embedding user-defined constraints into the algorithm. The most widely used are must-link constraints, which specify that certain pairs of entities must be placed in the same cluster, and cannot-link constraints, which require them to be placed in different clusters \cite{wagstaff2000clustering}. By integrating such pairwise constraints, the clustering task shifts from a purely unsupervised problem to a semisupervised (or constrained) clustering setting, where background information actively guides the cluster formation \cite{piccialli2022exact, piccialli2023global}. Extending this paradigm to biclustering allows prior knowledge to influence the simultaneous grouping of rows and columns, leading to constrained or semisupervised biclustering \cite{pensa2008constrained}.

Biclustering algorithms can be broadly categorized according to the bicluster structures they aim to detect \cite{busygin2008biclustering}. When multiple blocks are assumed to be present in the data matrix, common structures include: (i) overlapping biclusters, (ii) non-overlapping biclusters arranged in a checkerboard pattern, and (iii) biclusters with exclusive rows and columns forming a block-diagonal structure.  In this work, we focus on the block-diagonal case, where each row and column belongs to exactly one bicluster. This arrangement not only facilitates interpretation—since reordering rows and columns reveals the biclusters along the main diagonal—but also aligns naturally with the incorporation of background knowledge. For example, in gene expression analysis, block-diagonal biclusters allow for the clear categorization of genes into groups associated with specific cancer types, while in text mining they support the separation of documents into distinct clusters defined by unique sets of terms. Must-link and cannot-link constraints among rows or columns can be seamlessly integrated, directly shaping the block assignments in accordance with known biological, semantic, or domain-specific relationships. This synergy ensures that the resulting biclusters are both structurally coherent and meaningful for the application at hand.

%Biclustering is inherently linked to mathematical optimization and graph theory. Over the years, a large number of methods have been proposed, each one characterized by different attributes, such as computational complexity, interpretability, and optimization techniques. Comprehensive reviews of these methods can be found in \cite{madeira2004biclustering} and \cite{padilha2017systematic}. 

Since the seminal work of Cheng and Church \cite{cheng2000biclustering}, biclustering has been studied from numerous perspectives, including probabilistic models and mathematical programming formulations. Comprehensive reviews of these approaches can be found in \cite{madeira2004biclustering} and \cite{padilha2017systematic}. A particularly effective and well-studied class of methods formulates biclustering as a bipartite graph partitioning problem \cite{ding2006biclustering, fan2010integer, fan2012multi}. In this setting, the data matrix is represented as a weighted complete bipartite graph, where one vertex set corresponds to samples (rows) and the other to features (columns). Edge weights quantify the affinity between sample–feature pairs, and biclustering reduces to identifying complete bipartite subgraphs, called bicliques, that correspond to dense regions in the data matrix \cite{ames2014guaranteed, sudoso2024}. Numerous heuristic and decomposition-based methods have been proposed for unconstrained biclustering, including iterative row-column assignment strategies \cite{bergmann2003iterative, shabalin2009finding} and matrix factorization approaches, such as sparse and truncated singular value decomposition \cite{busygin2008biclustering, witten2009penalized, li2023beyond}. In contrast, constrained biclustering with must-link or cannot-link relations has received far less attention. According to the existing literature, the earliest attempt to address this problem is due to \cite{pensa2008constrained} and was later extended in \cite{pensa2010co} to incorporate a model that finds a bi-partition satisfying user-defined constraints while optimizing objective functions based on sum-of-squares residuals. Other constrained biclustering approaches build on non-negative matrix factorization models, where such constraints are enforced via penalty terms in the objective function \cite{chen2009non, song2010constrained}. Although typically fast, these methods do not guarantee global optimality and may produce biclusters that only approximately satisfy the side information. In particular, they may fail to strictly enforce all constraints and cannot ensure optimal solutions even when the constraints appear to be met.

Despite its potential, constrained biclustering remains an underdeveloped area from the standpoint of global optimization. Although exact approaches are not suited for large-scale datasets, global optimization plays a crucial role in constrained biclustering for two main reasons. First, it ensures that the identified biclusters are both representative of the underlying data patterns and consistent with the user-defined background knowledge, thereby providing guarantees on solution quality and feasibility. Such guarantees are particularly valuable in real-world applications, where heuristic methods may overlook high-quality solutions or violate constraints. %Moreover, as a machine learning task, constrained biclustering still requires interpretation by domain experts, and poor-quality solutions can lead to incorrect conclusions. 
Second, an exact solver serves as a reliable benchmark for evaluating and refining heuristic approaches, enabling the detection of potential shortcomings and guiding the development of more effective algorithms.

In this paper, we fill this gap by proposing the first exact algorithm for constrained biclustering. As a model problem, we focus on the constrained variant of the $k$-densest-disjoint biclique ($k$-DDB) problem \cite{ames2014guaranteed}, which incorporates must-link and cannot-link constraints. The objective is to identify $k$ disjoint bipartite complete subgraphs of the input graph that maximize the sum of their densities while satisfying all pairwise constraints. Leveraging semidefinite programming (SDP) tools, we develop an SDP-based branch-and-cut algorithm, extending the unconstrained solver recently proposed in \cite{sudoso2024}. This approach substantially increases the size of problems that can be solved to global optimality by standard software, highlighting its potential for practical applications. In addition, we design a scalable SDP-based heuristic capable of efficiently handling large-scale instances. The main contributions of the paper can be summarized as follows:

\begin{itemize}
    \item We propose a non-convex quadratically constrained quadratic programming (QCQP) formulation for the $k$-DDB problem pairwise constraints. The formulation exploits the properties of must-link and cannot-link constraints, leading to a substantial reduction in both the number of variables and constraints. For small graphs, we show that this model can be solved to global optimality using general-purpose global optimization tools.
    
    \item To address larger graphs, we design an exact algorithm based on the branch-and-cut technique. Bounds are obtained from a low-dimensional SDP relaxation of the QCQP model, strengthened with valid inequalities and solved via a first-order method in a cutting-plane fashion. Exploiting integer programming tools, a rounding scheme transforms SDP solutions into feasible bicliques.

    \item Leveraging established optimization techniques, we introduce a heuristic based on a low-rank factorization of the SDP relaxation that exploits the underlying problem structure. The resulting non-convex formulation is tackled through an augmented Lagrangian method, where subproblems are solved via a block-coordinate projected gradient algorithm in a Gauss–Seidel scheme.

\end{itemize}

Extensive computational experiments on synthetic and real-world datasets show that the exact algorithm can handle graphs with a number of nodes ten times larger than those addressed by general-purpose solvers, thus serving as both a benchmark tool and a certifier of solution quality. The heuristic algorithm can efficiently tackle large-scale instances, delivering high-quality solutions not only in terms of objective value but also according to external machine learning validation metrics. This work highlights how mathematical optimization can effectively address complex machine learning problems, advancing the state of the art and reinforcing the value of cross-fertilization between these fields. To foster reproducibility and further research, the source code of the proposed solvers is publicly available at \url{https://github.com/antoniosudoso/cbicl}.

The remainder of the paper is organized as follows. Section \ref{section:problem_formulation} introduces key concepts from biclustering and graph theory, and then presents the mathematical programming formulation. Section \ref{section:branch-and-cut} details the components of the proposed branch-and-cut method. Section \ref{section:low_rank} describes the low-rank factorization heuristic. Section \ref{section:computational_results} provides implementation details and presents computational results on both synthetic and real-world datasets. Finally, Section \ref{section:conclusions} summarizes the findings and outlines directions for future research.

\subsubsection*{Notation}
%Throughout this work, we use the notation $\mathbb{S}^n$ to represent the set of $n \times n$ real symmetric matrices, while $\mathbb{R}^n$ and $\mathbb{R}^{m \times n}$ denote the spaces of real-valued vectors of dimension $n$ and real-valued matrices of size $m \times n$, respectively. A matrix ${X}$ is said to be positive semidefinite if ${X} \succeq 0$, and we let $\mathbb{S}^n_+$ be the cone of all such $n \times n$ positive semidefinite matrices. The vectors ${0_n}$ and ${1_n}$ refer to the $n$-dimensional vectors of all zeros and all ones, respectively. $\Pi_C(x)$ denotes the Euclidean projection of the vector/matrix $x$ onto the set $C$. 
%For a vector ${x} \in \mathbb{R}^n$, we denote by $\textrm{Diag}({x})$ the diagonal matrix whose diagonal entries are given by ${x}$. Conversely, for a matrix ${X} \in \mathbb{R}^{n \times n}$, $\textrm{diag}({X})$ denotes the vector composed of the diagonal elements of ${X}$. The Frobenius inner product is denoted as $\inprod{{A}}{{B}} := \textrm{tr}({B}^\top {A})$, for any ${A}, {B} \in \mathbb{R}^{m\times n}$, and the Frobenius norm of a matrix $X$ is denoted as $\|A\|_F = \sqrt{\inprod{A}{A}}$. Finally, the spectrum of a matrix ${A} \in \mathbb{R}^{n \times n}$ is denoted by $\lambda({A})$, representing the set of its eigenvalues.

Throughout this work, $\mathbb{S}^n$ denotes the set of $n \times n$ real symmetric matrices; $\mathbb{S}_+^n$ is the cone of positive semidefinite matrices; $\mathbb{R}^n$ and $\mathbb{R}^{m \times n}$ are the spaces of real vectors and matrices, respectively. ${X} \succeq 0$ denotes a matrix ${X} \in \mathbb{S}^n$ that is positive semidefinite. Vectors $0_n$ and $1_n$ denote the all-zeros and all-ones vectors in $\mathbb{R}^n$; subscripts are omitted when the size is clear from the context. For $x \in \mathbb{R}^n$, $\mathrm{Diag}(x)$ is the diagonal matrix with entries from $x$, while for $X \in \mathbb{R}^{n \times n}$, $\mathrm{diag}(X)$ is the vector of its diagonal entries. The Frobenius inner product is $\langle A, B \rangle := \mathrm{tr}(B^\top A)$, with norm $\|A\|_F = \sqrt{\langle A, A \rangle}$. $\Pi_C(x)$ denotes the Euclidean projection of $x$ onto a set $C$. Finally, $\lambda(A)$ denotes the set of eigenvalues of $A \in \mathbb{R}^{n \times n}$.

\section{Problem formulation}\label{section:problem_formulation}

Data for biclustering is stored in a rectangular matrix with $n$ rows and $m$ columns. Let $A \in \mathbb{R}^{n \times m}$ represent such a matrix, where the rows are indexed by the set $R = \{r_1, \dots, r_n\}$ and the columns by the set $C = \{c_1, \dots, c_m\}$. Here, the entry $A_{ij}$ represents the relationship between row $r_i \in R$ and column $c_j \in C$. The definition of biclustering can be formalized as follows \citep{madeira2004biclustering, busygin2008biclustering}. Given an integer \( k \), a {biclustering} of \( A \) is a collection of subsets \( \{(R_1, C_1), \dots, (R_k, C_k)\} \), where \( R_i \subseteq R \) and \( C_i \subseteq C \) for all \( i \in \{1, \dots, k\} \), such that \( \{R_1, \dots, R_k\} \) forms a partition of \( R \), and \( \{C_1, \dots, C_k\} \) forms a partition of \( C \). Each pair \( (R_i, C_i) \) defines a \emph{bicluster}, representing a submatrix where the rows in \( R_i \) are assumed to exhibit similar patterns across the columns in \( C_i \).

\begin{comment}
\begin{definition}
Given a matrix \( A \in \mathbb{R}^{n \times m} \) and an integer \( k \), a \emph{biclustering} of \( A \) is a collection of subsets \( \{(R_1, C_1), \dots, (R_k, C_k)\} \), where \( R_i \subseteq R \) and \( C_i \subseteq C \) for all \( i \in \{1, \dots, k\} \), such that \( \{R_1, \dots, R_k\} \) forms a partition of \( R \), and \( \{C_1, \dots, C_k\} \) forms a partition of \( C \). Each pair \( (R_i, C_i) \) defines a \emph{bicluster}, representing a submatrix where the rows in \( R_i \) are assumed to exhibit similar patterns across the columns in \( C_i \).
\end{definition}
\end{comment}

In practical applications, background knowledge in the form of pairwise constraints can be leveraged to guide the biclustering process. %One common approach is to use pairwise constraints, which impose specific requirements on how certain rows and columns should be grouped or separated during biclustering. 
These constraints are typically categorized as must-link and cannot-link, and are defined as follows \cite{pensa2010co}. If rows $r_i, r_{i'} \in R$ (resp. columns $c_j, c_{j'} \in C$) are involved in a \textit{must-link} constraint, denoted $l_{=}(r_i, r_{i'})$ (resp. $l_{=}(c_i, c_{i'})$), they must be in the same row cluster $R_h$ (resp. column cluster $C_h$) for some $h \in \{1, \dots, k\}$. If rows $r_i, r_{i'} \in R$ (resp. columns $c_j, c_{j'} \in C$) are involved in a \textit{cannot-link} constraint, denoted $l_{\neq}(r_i, r_{i'})$ (resp. $l_{\neq}(c_i, c_{i'})$), they cannot be in the same row cluster $R_h$ (resp. column cluster $C_h$) for any $h \in \{1, \dots, k\}$.

\begin{comment}
\begin{definition}[Must-link/cannot-link constraints]
    If rows $r_i, r_{i'} \in R$ (resp. columns $c_j, c_{j'} \in C$) are involved in a \textit{must-link} constraint, denoted $l_{=}(r_i, r_{i'})$ (resp. $l_{=}(c_i, c_{i'})$), they must be in the same row cluster $R_h$ (resp. column cluster $C_h$) for some $h \in \{1, \dots, k\}$. If rows $r_i, r_{i'} \in R$ (resp. columns $c_j, c_{j'} \in C$) are involved in a \textit{cannot-link} constraint, denoted $l_{\neq}(r_i, r_{i'})$ (resp. $l_{\neq}(c_i, c_{i'})$), they cannot be in the same row cluster $R_h$ (resp. column cluster $C_h$) for any $h \in \{1, \dots, k\}$.
\end{definition}
\end{comment}

Constrained biclustering can be formulated as a mathematical optimization problem. Given the number of biclusters $k$, the objective is to determine the optimal biclustering that maximizes a predefined quality function while satisfying all specified constraints. Formally, the problem is stated as: 
\begin{align*} 
\max \{ f(B) : B \in \mathcal{P}(A, k, l_{=}, l_{\neq})\} ,
\end{align*}
where $\mathcal{P}(A, k, l_{=}, l_{\neq})$ is the set of all possible collections of size $k$ containing row and column subsets of $A$ that adhere to the pairwise constraints, and $f: \mathcal{P}(A, k, l_{=}, l_{\neq}) \rightarrow \mathbb{R}$ is the objective function that evaluates the quality of the discovered biclusters. In the following, using elements of graph theory, we cast this task as a partitioning problem on a bipartite graph. 

Given a bipartite graph $G = ((U \cup V), E)$, a pair of disjoint subsets $U' \subseteq U$, $V' \subseteq V$ defines a \textit{biclique} if every node in $U'$ is connected to every node in $V'$. The subgraph induced by $(U' \cup V')$, denoted $B' = G[(U' \cup V')]$, is thus a complete bipartite subgraph. Let \( K_{n, m} = (U \cup V, E) \) be a weighted complete bipartite graph, where \( U = \{u_1, \dots, u_n\} \) represents the rows and \( V = \{v_1, \dots, v_m\} \) the columns of matrix $A$. The weight function \( w : U \times V \rightarrow \mathbb{R} \) assigns a real-valued weight to each edge. That is, there is an edge $(u_i, v_j) \in E$ with weight $w(u_i, v_j) = A_{ij}$ between each pair of vertices $u_i \in U$ and $v_j \in V$. In this setting, any biclique $(U' \cup V')$ corresponds to a \textit{row cluster} $U' \subseteq U$ and a \textit{column cluster} $V' \subseteq V$, and thus to a bicluster of the matrix. Due to this correspondence, we use the terms ``bicluster'' and ``biclique'' interchangeably. Given a desired number of biclusters \( k \in \{2, \dots, \min\{|U|, |V|\}\} \), the \( k \)-DDB problem seeks to identify a collection of \( k \) disjoint bicliques \( \{(U_1 \cup V_1), \dots, (U_k \cup V_k)\} \) in \( K_{n, m} \) that maximizes the sum of the densities of the corresponding induced subgraphs \citep{ames2014guaranteed}. The density of a subgraph $H = ((U' \cup V'), E')$ of $K_{n,m}$, denoted by $d_H$, is defined as the total edge weight incident at each vertex divided by the square root of the number of edges.

The $k$-DDB problem can be expressed as a continuous QCQP with an underlying discrete structure \citep{ames2014guaranteed, sudoso2024}. Let \( X_U \in \{0, 1\}^{n \times k} \) be the {row partition matrix}, where the \(i\)-th row encodes the cluster assignment of vertex \( u_i \in U \), and the \(j\)-th column is the characteristic vector of the cluster \( U_j \). That is, \( (X_U)_{ij} = 1 \) if \( u_i \in U_j \), and 0 otherwise. Analogously, define \( X_V \in \{0, 1\}^{m \times k} \) as the {column partition matrix} for the vertices \( v_j \in V \). The resulting partition matrices satisfy the constraints $X_U 1_k = 1_n$ and $X_V 1_k = 1_m$, ensuring that each vertex in $U$ and $V$ is assigned to exactly one row cluster and column cluster, respectively. 
%Let ${X_U} \in \{0, 1\}^{n \times k}$ be the partition matrix where the $i$-th row represents the cluster assignment for $u_i \in U$ and the $j$-th column is the characteristic vector of $U_j$, i.e., $({X_U})_{ij} = 1$ if $u_i \in U_j$ and 0 otherwise. Similarly, let ${X_V} \in \{0, 1\}^{m \times k}$ be the partition matrix where the $i$-th row represents the cluster assignment for $v_i \in V$ and the $j$-th column is the characteristic vector of $V_j$, i.e., $({X_V})_{ij} = 1$ if $v_i \in V_j$ and 0 otherwise. The resulting partition matrices $X_U$ and $X_V$ satisfy the constraints $X_U 1_k = 1_n$ and $X_V 1_k = 1_m$, ensuring that each object in $U$ and $V$ is assigned to exactly one row cluster and column cluster, respectively. 
For a collection of \(k\) disjoint bicliques \( \{(U_1 \cup V_1), \dots, (U_k \cup V_k)\} \), let \( B_j = K_{n,m}[(U_j \cup V_j)] \) denote the complete bipartite subgraph induced by the \(j\)-th biclique. The total density of the biclique collection can be written as:
\[
\sum_{j=1}^k d_{B_j} = \sum_{j=1}^k \frac{\sum_{u \in U_j, v \in V_j} w(u, v)}{\sqrt{|U_j||V_j|}} = \tr\left(({X_U^\top X_U})^{-\frac{1}{2}} X_U^\top A X_V ({X_V^\top X_V})^{-\frac{1}{2}}\right).
\]
To simplify the notation and handle the normalization terms, we introduce diagonal matrices \( P_U = (X_U^\top X_U)^{-1/2} = \Diag(1/\sqrt{|U_1|}, \dots, 1/\sqrt{|U_k|}) \) and \( P_V = (X_V^\top X_V)^{-1/2} = \Diag(1/\sqrt{|V_1|}, \dots, 1/\sqrt{|V_k|}) \), and define the normalized matrices $Y_U = X_U P_U$ and $Y_V = X_V P_V$.
This change of variables leads to the following formulation of the \(k\)-DDB problem as a QCQP:
\begin{subequations}
\label{prob:or}
\begin{align}
\max \quad & \tr(Y_U^\top A {Y_V})  \\
\textrm{s.\,t.} \quad & Y_U^\top {Y_U} = {I_k}, \ Y_V^\top {Y_V} = {I_k}, \label{constr:ort}\\
& Y_U Y_U^\top 1_n = 1_n, \ Y_V Y_V^\top 1_m = 1_m, \label{constr:ort_sum}\\
& {Y_U} \geq {0}_{n \times k}, \ {Y_V} \geq {0}_{m \times k}. \label{constr:nonneg}
\end{align}
\end{subequations}
Although Problem \eqref{prob:or} is a continuous formulation, the feasible set enforces a solution structure that is implicitly discrete. Specifically, the orthogonality constraints \eqref{constr:ort} imply that $Y_U$ and $Y_V$ are matrices with orthonormal columns. Using $Y_U \geq 0$, we obtain that $(Y_U)_{ij} \neq 0$ implies that $(Y_U)_{ih} = 0$ for all $h \neq j$, so $Y_U$ has exactly one nonnegative entry per row. Similarly, using $Y_V \geq 0$, we obtain that $(Y_V)_{ij} \neq 0$ implies that $(Y_V)_{ih} = 0$ for all $h \neq j$. %This restricts each row of $Y_U$ and $Y_V$ to have exactly one nonzero entry—mirroring the behavior of hard cluster assignments. 
The constraint $Y_U Y_U^\top 1_n = 1_n$ implies that the vector $1_n$ belongs to the span of the columns of $Y_U$. Therefore if $(Y_U)_{ij} \neq 0$ and $(Y_U)_{hj} \neq 0$ then $(Y_U)_{ij} = (Y_U)_{hj}$. Similarly, the constraint $Y_V Y_V^\top 1_m = 1_,$ implies that the vector $1_m$ belongs to the span of the columns of $Y_V$. 
Hence, constraints \eqref{constr:ort}, \eqref{constr:ort_sum}, \eqref{constr:nonneg} ensure that the columns of $Y_U$ and $Y_V$ can only assume two values: $(Y_U)_{ij} \in \{0, |U_j|^{-\frac{1}{2}}\}$ and $(Y_V)_{ij} \in \{0, |V_j|^{-\frac{1}{2}}\}$. Therefore, using this characterization, it is easy to verify that 
\begin{enumerate}
    \item If $u_i, u_j \in U$ (resp. $v_i, v_j \in V$) are in the same row cluster (resp. column cluster), then $(Y_U)_{ih} = (Y_U)_{jh}$ (resp. $(Y_V)_{ih} = (Y_V)_{jh}$) for all $h \in \{1, \dots, k\}$.
    \item If $u_i, u_j \in U$ (resp. $v_i, v_j \in V$) are not in the same row cluster (resp. column cluster), then $(Y_U Y_U^\top)_{ij} = 0$ (resp. $(Y_V Y_V^\top)_{ij} = 0$).
\end{enumerate}
To incorporate background knowledge in the form of pairwise constraints, we introduce the following sets. Denote by $\textrm{ML}_U \subseteq U \times U$ (resp. $\textrm{ML}_V \subseteq V \times V$) the set of must-link constraints between vertices in $U$ (resp. $V$). Furthermore, let $\textrm{CL}_U \subseteq U \times U$ (resp. $\textrm{CL}_V \subseteq V \times V$) be the set of cannot-link constraints between vertices in $U$ (resp. $V$). The $k$-DDB problem with pairwise constraints can be formulated as
\begin{subequations}
\label{prob:or_constr}
\begin{alignat}{2}
\max \quad & \tr(Y_U^\top A {Y_V})  && \\
\textrm{s.\,t.} \quad & Y_U^\top {Y_U} = {I_k}, \ Y_U Y_U^\top 1_n = 1_n, && \\
& {(Y_U)_{ih} - (Y_U)_{jh} = 0} && \quad \forall h \in \{1, \dots, k\}, \ \forall (u_i, u_j) \in \textrm{ML}_U, \label{constr:ml_U}\\
& (Y_U Y_U^\top)_{ij} = 0 && \quad \forall (u_i, u_j) \in \textrm{CL}_U, \label{constr:cl_U}\\
& Y_V^\top {Y_V} = {I_k}, \ Y_V Y_V^\top 1_m = 1_m && \label{constr:ort_sum_constr}\\
& {(Y_V)_{ih} - (Y_V)_{jh} = 0} && \quad \forall h \in \{1, \dots, k\}, \ \forall (v_i, v_j) \in \textrm{ML}_V, \label{constr:ml_V}\\
& (Y_V Y_V^\top)_{ij} = 0 && \quad \forall (v_i, v_j) \in \textrm{CL}_V, \label{constr:cl_V} \\
& {Y_U} \geq {0}_{n \times k}, \ {Y_V} \geq {0}_{m \times k}. && \label{constr:nonneg_constr}
\end{alignat}
\end{subequations}
Here, constraints \eqref{constr:ml_U} and \eqref{constr:ml_V} ensure that must-linked vertices share identical cluster memberships, while \eqref{constr:cl_U} and \eqref{constr:cl_V} enforce that cannot-linked vertices belong to different clusters. The other constraints preserve the structure of Problem~\eqref{prob:or}.
As is well known, pairwise constraints have several useful properties~\cite{wagstaff2000clustering, basu2008constrained}. In particular, must-link constraints define an equivalence relation—symmetric, reflexive, and transitive—on the vertices involved. To exploit this, let \( G^{\mathrm{ML}} = (U \cup V, \mathrm{ML}_U \cup \mathrm{ML}_V) \) denote the undirected graph defined by the must-link constraints. Its transitive closure decomposes the vertex set into a disjoint union of cliques, or connected components, each of which must be assigned to the same bicluster. This allows to treat each component into a single aggregated vertex, thereby reducing the problem size. To formalize this, let \( G_U^{\mathrm{ML}} = (U, \mathrm{ML}_U) \) and \( G_V^{\mathrm{ML}} = (V, \mathrm{ML}_V) \) be the subgraphs of \( G^{\mathrm{ML}} \) induced by \( U \) and \( V \), respectively. Denote their connected components by \( \bar{U} = \{ \mathcal{U}_1, \dots, \mathcal{U}_{\bar{n}} \} \) and \( \bar{V} = \{ \mathcal{V}_1, \dots, \mathcal{V}_{\bar{m}} \} \), where \( \bar{n} \leq n \) and \( \bar{m} \leq m \). Define the indicator matrices \( T_U \in \{0,1\}^{\bar{n} \times n} \) and \( T_V \in \{0,1\}^{\bar{m} \times m} \), with entries \( (T_U)_{ij} = 1 \) if \( u_j \in \mathcal{U}_i \), and similarly \( (T_V)_{ij} = 1 \) if $v_j \in \mathcal{V}_i$. We refer to \( T_U \) and \( T_V \) as the \emph{must-link matrices}, which identify the component memberships of the original vertices. The vectors \( e_U = T_U {1}_n \in \mathbb{R}^{\bar{n}} \) and \( e_V = T_V {1}_m \in \mathbb{R}^{\bar{m}} \) count the number of vertices in each component. By construction, these satisfy the identities $\mathrm{Diag}(e_U) = T_U T_U^\top$ and $\mathrm{Diag}(e_V) = T_V T_V^\top$.
Although cannot-link constraints do not define an equivalence relation, they interact naturally with must-link constraints. Specifically, if two vertices in different connected components are in a cannot-link relationship, then the entire components must not be assigned to the same bicluster. It is thus sufficient to enforce cannot-link constraints between aggregated vertices, resulting in a reduced but equivalent constraint set. Formally, we define the \emph{projected cannot-link sets} as $\overline{\textrm{CL}}_U = \{(\mathcal{U}_i, \mathcal{U}_j) \in \bar{U} \times \bar{U} :  (u_h, u_{h'}) \in \textrm{CL}_U, \, u_h \in \mathcal{U}_i, \ u_{h'} \in \mathcal{U}_j\}$ and $ \overline{\textrm{CL}}_V = \{(\mathcal{V}_i, \mathcal{V}_j) \in \bar{V} \times \bar{V} :  (v_h, v_{h'}) \in \textrm{CL}_V, \, v_h \in \mathcal{V}_i, \ v_{h'} \in \mathcal{V}_j\}$. 

Problem~\eqref{prob:or_constr} can be reformulated as an equivalent QCQP whose decision variables index only the $\bar n+\bar m$ aggregated vertices and whose cannot-link constraints are directly enforced between connected components. This yields a significant reduction in both the number of variables and the number of constraints. Indeed, we can state the constrained $k$-DDB problem as
\begin{subequations}
\label{prob:or_constr_shr}
\begin{alignat}{2}
\max \quad & \tr(\bar{Y}_U^\top T_U A T_V^\top \bar{Y}_V)  && \\
\textrm{s.\,t.} \quad & \bar{Y}_U^\top \Diag(e_U) \bar{Y}_U = {I_k}, \ \bar{Y}_U \bar{Y}_U^\top e_U = 1_{\bar{n}}, && \\
& (\bar{Y}_U \bar{Y}_U^\top)_{ij} = 0 && \quad \forall (\mathcal{U}_i, \mathcal{U}_j) \in \overline{\textrm{CL}}_U, \label{constr:cl_U_shr}\\
& \bar{Y}_V^\top \Diag(e_V)  \bar{Y}_V = {I_k}, \ \bar{Y}_V \bar{Y}_V^\top e_V = 1_{\bar{m}} && \label{constr:ort_sum_constr_shr}\\
& (\bar{Y}_V \bar{Y}_V^\top)_{ij} = 0 && \quad \forall (\mathcal{V}_i,\mathcal{V}_j) \in \overline{\textrm{CL}}_V, \label{constr:cl_V_shr} \\
& \bar{Y}_U \geq {0}_{\bar{n} \times k}, \ \bar{Y}_V \geq {0}_{\bar{m} \times k}. && \label{constr:nonneg_constr_shr}
\end{alignat}
\end{subequations}
In the reduced formulation, orthogonality and normalization constraints preserve the bicluster structure over aggregated vertices, while constraints \eqref{constr:cl_U_shr} and \eqref{constr:cl_V_shr} prevent any merged components connected by cannot-links from being assigned to the same bicluster.

\begin{proposition}\label{proposition:equivalence_shr}
    Problems \eqref{prob:or_constr} and \eqref{prob:or_constr_shr} are equivalent.
\end{proposition}

\vspace{-15pt}

\begin{proof}
See Appendix \ref{appendix:equivalence}.
\end{proof}

To compute the globally optimal solution of Problem \eqref{prob:or_constr_shr}, one can employ spatial branch-and-bound algorithms implemented in general-purpose solvers like Gurobi and BARON \citep{sahinidis1996baron}. However, solving non-convex QCQP problems remains particularly challenging: despite advances in global optimization, the size of instances that can be solved to provable optimality is still very limited \cite{chen2012globally, locatelli2024fix}. In practice, even for relatively small graphs, off-the-shelf exact solvers often require substantial computational time to certify global optimality. Although the reformulated QCQP in \eqref{prob:or_constr_shr} reduces both the dimensionality and the number of constraints relative to the original problem, it remains too computationally demanding to scale to the medium- and large-scale datasets typical of real-world applications. To overcome this limitation, the next section presents a tailored branch-and-cut solver. %designed to efficiently handle significantly larger instances than those solvable by general-purpose optimization tools.

\section{Branch-and-cut algorithm}\label{section:branch-and-cut}
In this section, we present the key components of the exact algorithm for solving Problem \eqref{prob:or_constr_shr}. We begin with an equivalent rank-constrained SDP formulation, which we then relax to obtain a tractable SDP. Next, we describe the upper bounding procedure, which leverages valid inequalities within a cutting-plane framework to compute tight upper bounds. Finally, we introduce a rounding heuristic that exploits the solution of the SDP relaxation to construct feasible bicliques.

\subsection{Upper bound computation}\label{subsection:upper_bound}
Similar to its unconstrained variant, Problem \eqref{prob:or_constr_shr} can be reformulated as a rank-constrained SDP problem. To see this, we lift the variables \( \bar{Y}_U \) and \( \bar{Y}_V \) into the space of \( (\bar{n} + \bar{m}) \times (\bar{n} + \bar{m}) \) matrices by introducing a symmetric block matrix
\begin{align*}%\label{eq:Z_block}
    \bar{Z} = \begin{bmatrix}
{\bar{Y}_U} \bar{Y}_U^\top & {\bar{Y}_U} \bar{Y}_V^\top\\
{\bar{Y}_V} \bar{Y}_U^\top & {\bar{Y}_V} \bar{Y}_V^\top
\end{bmatrix} \in \mathbb{S}^{\bar{n}+\bar{m}}.
\end{align*}
Observe that, by construction, $\rank(\bar{Z}) = k$, implying that the constrained $k$-DDB problem has a low-rank structure. Moreover $\bar{Z} \succeq 0$ and $\bar{Z} \geq 0$ hold as well. In line with standard techniques for deriving SDP relaxations of QCQPs, we linearize the quadratic terms by applying the change of variable ${Z_{UU}} = {\bar{Y}_U} \bar{Y}_U^\top    \in \mathbb{S}^{n}$, $Z_{VV} = {\bar{Y}_V} \bar{Y}_V^\top \in \mathbb{S}^{m}$ and ${Z_{UV}} = {\bar{Y}_U} \bar{Y}_V^\top \in \mathbb{R}^{\bar{n} \times \bar{m}}$. One can easily verify that constraints $Z_{UU} e_U = 1_{\bar{n}}$, $\inprod{\Diag(e_U)}{Z_{UU}} = k$, $Z_{VV} e_V = 1_{\bar{m}}$ and $\inprod{\Diag(e_V)}{Z_{VV}} = k$ hold. Thus, the SDP relaxation for the constrained $k$-DDB problem can be expressed as
\begin{subequations}
\label{prob:sdp_shr}
\begin{alignat}{2}
\max~ \quad & \inprod{T_U A T_V^\top}{Z_{UV}} \\
\st~ \quad & Z_{UU} e_U = 1_{\bar{n}}, \  \inprod{\Diag(e_U)}{Z_{UU}} = k,\\
& (Z_{UU})_{ij} = 0, && \quad \forall (\mathcal{U}_i, \mathcal{U}_j) \in \overline{\textrm{CL}}_U, \\
& Z_{VV} e_V = 1_{\bar{m}}, \ \inprod{\Diag(e_V)}{Z_{VV}} = k, \\
& (Z_{VV})_{ij} = 0, && \quad \forall (\mathcal{V}_i, \mathcal{V}_j) \in \overline{\textrm{CL}}_V, \\
& 
Z = \begin{bmatrix}
Z_{UU} & Z_{UV}\\
Z_{UV}^\top & Z_{VV}
\end{bmatrix}, \ Z \in \mathbb{S}_+^{\bar{n}+\bar{m}}, \ Z \geq 0.
\end{alignat}
\end{subequations}
Problem \eqref{prob:sdp_shr} is a double nonnegative program (DNN) since the block matrix ${Z}$ is both positive semidefinite and elementwise nonnegative.
Clearly, the optimal objective function value of the SDP relaxation provides an upper bound on the optimal value of the original non-convex problem. Furthermore, if the optimal solution of the relaxation has rank equal to $k$, then it is also optimal for Problem \eqref{prob:or_constr_shr}.

Additional linear constraints can be imposed to strengthen the bound. De Rosa and Khajavirad~\cite{de2022ratio} studied the ratio-cut polytope, defined as the convex hull of ratio-cut vectors associated with all partitions of \(n\) points into at most \(k\) clusters. 
This polytope, which is closely related to the convex hull of the feasible region of the constrained \(k\)-DDB problem, was analyzed to characterize its facial structure. 
The authors derived several families of facet-defining inequalities, including the well-known pair and triangle inequalities, expressed as:

\newpage

\begin{align}
    ({Z_{UU}})_{ij} & \leq ({Z_{UU}})_{ii} &&\forall \ \textrm{distinct} \ \mathcal{U}_i, \mathcal{U}_j \in \bar{U}\label{eq:pair_u},\\
    ({Z_{VV}})_{ij} & \leq ({Z_{VV}})_{ii} &&\forall \ \textrm{distinct} \ \mathcal{V}_i, \mathcal{V}_j \in \bar{V}\label{eq:pair_v}, \\
({Z_{UU}})_{ij} + ({Z_{UU}})_{ih} & \leq ({Z_{UU}})_{ii} + ({Z_{UU}})_{jh} &&\forall \ \textrm{distinct} \ \mathcal{U}_i, \mathcal{U}_j, \mathcal{U}_h \in \bar{U} \label{eq:tri_u},\\
({Z_{VV}})_{ij} + ({Z_{VV}})_{ih} & \leq ({Z_{VV}})_{ii} + ({Z_{VV}})_{jh} &&\forall \ \textrm{distinct} \ \mathcal{V}_i, \mathcal{V}_j, \mathcal{V}_h \in \bar{V}.\label{eq:tri_v}
\end{align}
These inequalities exploit the following key property in clustering problems.
Consider three vertices $\mathcal{U}_i, \mathcal{U}_j, \mathcal{U}_h \in \bar{U}$ (resp. $\mathcal{V}_i, \mathcal{V}_j, \mathcal{V}_h \in \bar{V}$). If $(\mathcal{U}_i, \mathcal{U}_j)$ and $(\mathcal{U}_j, \mathcal{U}_h)$ (resp. $(\mathcal{V}_i, \mathcal{V}_j)$ and $(\mathcal{V}_j, \mathcal{V}_h)$) are in the same row cluster $\bar{U}' \subseteq \bar{U}$ (resp. column cluster $\bar{V}' \subseteq \bar{V}$), then $(\mathcal{U}_i, \mathcal{U}_h)$ (resp. $(\mathcal{V}_i, \mathcal{V}_h)$) must be in $\bar{U}'$ (resp. $\bar{V}'$). 

To incorporate such inequalities into the SDP relaxation, we implement a cutting-plane algorithm in which violated pairs and triangles are iteratively added and removed from the SDP. Although the separation problem for these inequalities can be solved exactly in $O(\bar{n}^3 + \bar{m}^3)$ time via enumeration, this becomes inefficient for large graphs. To address this, we use a heuristic separation routine that generates only a fixed number of violated inequalities at each iteration. In each iteration of our cutting-plane approach, valid inequalities are added and purged after computing an upper bound. We begin by solving the basic SDP relaxation in \eqref{prob:sdp_shr}, then add an initial set of valid inequalities, prioritizing the most violated ones. After solving the updated problem, we remove all inactive constraints and introduce newly violated inequalities. This process is repeated until the upper bound no longer improves.

While polynomial-time algorithms are available for solving SDPs, general-purpose interior-point methods often struggle with scalability. To efficiently compute upper bounds for large-scale instances—particularly in the context of branch-and-bound—a practical alternative is to apply a first-order method to obtain an approximate dual solution of the SDP relaxation, followed by a suitable post-processing step to extract valid bounds \cite{jansson2008rigorous, yang2015sdpnal}. To derive the dual SDP, we first rewrite the primal in compact form. We collect all the equality constraints by defining two operators $\mathcal{A}_U : \mathbb{S}^{\bar{n}}\rightarrow \mathbb{R}^{\bar{n}+|\overline{\textrm{CL}}_U| + 1}$ and $\mathcal{A}_V: \mathbb{S}^{\bar{m}}\rightarrow \mathbb{R}^{\bar{m}+|\overline{\textrm{CL}}_V| + 1}$ and vectors $b_U \in \mathbb{R}^{\bar{n}+|\overline{\textrm{CL}}_U| + 1}$ and $b_V \in \mathbb{R}^{\bar{m}+|\overline{\textrm{CL}}_V| + 1}$. Furthermore, we collect valid inequalities \eqref{eq:pair_u}--\eqref{eq:tri_v} by using the operators $\mathcal{B}_U : \mathbb{S}^{\bar{n}} \rightarrow \mathbb{R}^{p}$ and $\mathcal{B}_V : \mathbb{S}^{\bar{m}} \rightarrow \mathbb{R}^{q}$. The primal SDP can be rewritten as
\begin{comment}
\begin{align}\label{prob:primal}
    Z^\star \in \argmax_{} \left\{\frac{1}{2}\tr(WZ) \ : \ \mathcal{A}(Z) = 1_{n+m}, \ \tr(Z) = 2k, \ \mathcal{B}(Z) \leq 0_q, \ Z \geq 0, \ Z \in \mathbb{S}_+^{n+m} \right \}
\end{align}
\end{comment}
\begin{equation}
\label{prob:primal}
\begin{aligned}
\max \quad & \inprod{T_U A T_V^\top}{Z_{UV}} \\
\textrm{s.t.} \quad & \mathcal{A}_U({Z_{UU}}) = b_U, \ \mathcal{B}_U({Z_{UU}}) \leq {0_q}, \\
& \mathcal{A}_V({Z_{VV}}) = b_V, \ \mathcal{B}_V({Z_{VV}}) \leq {0_p}, \\
& Z = \begin{bmatrix}
Z_{UU} & Z_{UV}\\
Z_{UV}^\top & Z_{VV}
\end{bmatrix}, \ Z \succeq 0, \ Z \geq 0.
\end{aligned}
\end{equation}
Consider Lagrange multipliers ${y_U} \in \mathbb{R}^{\bar{n}}$, ${y_V} \in \mathbb{R}^{\bar{m}}$, $\alpha_U, \alpha_V \in \mathbb{R}$, ${t}_U \in \mathbb{R}^p$, ${t}_V \in \mathbb{R}^q$, $\tau_U \in \mathbb{R}^{|\overline{\textrm{CL}}_V|}$, $\tau_V \in \mathbb{R}^{|\overline{\textrm{CL}}_V|}$, ${Q}$, ${S} \in \mathbb{S}^{\bar{n}+\bar{m}}$. Let $\lambda_U = [y_U; \alpha_U; \tau_U] \in \mathbb{R}^{\bar{n} + 1 + |\overline{\textrm{CL}}_U|}$ and $ \lambda_V = [y_V; \alpha_V; \tau_V] \in \mathbb{R}^{\bar{m} + 1 + |\overline{\textrm{CL}}_V|}$. Then, using the standard derivation in Lagrangian duality theory \cite{boyd2004convex}, the dual of Problem \eqref{prob:primal} is given by
\begin{equation}
\label{prob:dual}
\begin{aligned}
\min \quad & b_U^\top \lambda_U + b_V^\top \lambda_V := {y_U^\top} 1_{\bar{n}} + y_V^\top 1_{\bar{m}} + k(\alpha_U + \alpha_V) \\
\textrm{s.t.} \quad & T_U A T_V^\top + Q_{UV} + S_{UV} = 0_{\bar{n} \times \bar{m}}, \\
& - \mathcal{A}_U^\top(\lambda_U) - \mathcal{B}^\top(t_U) + Q_{UU} + S_{UU} = 0_{\bar{n} \times \bar{n}}, \\
& - \mathcal{A}_V^\top(\lambda_V) - \mathcal{B}^\top(t_V) + Q_{VV} + S_{VV} = 0_{\bar{m} \times \bar{m}}, \\
& Q =\begin{bmatrix}
{Q_{UU}} & {Q_{UV}}\\
Q_{UV}^\top & {Q_{VV}}
\end{bmatrix} \geq 0, \ S =\begin{bmatrix}
{S_{UU}} & {S_{UV}}\\
S_{UV}^\top & {S_{VV}}
\end{bmatrix} \succeq 0, \ t_U \geq 0_p, \ t_V \geq 0_q,
\end{aligned}
\end{equation}
where $\mathcal{A}_U^\top : \mathbb{R}^{\bar{n}+|\overline{\textrm{CL}}_U| + 1} \rightarrow\mathbb{S}^{\bar{n}}$, $\mathcal{A}_V^\top : \mathbb{R}^{\bar{m}+|\overline{\textrm{CL}}_V| + 1} \rightarrow\mathbb{S}^{\bar{m}}$, $\mathcal{B}_U^\top$ and $\mathcal{B}_V^\top$ are the adjoint operators of $\mathcal{A}_U$, $\mathcal{A}_V$, $\mathcal{B}_U^\top : \mathbb{R}^{p} \rightarrow\mathbb{S}^{\bar{n}}$ and $\mathcal{B}_V^\top : \mathbb{R}^{q} \rightarrow\mathbb{S}^{\bar{m}}$, respectively. An approximate solution to the dual problem is used to compute a valid upper bounds. 
\begin{comment}
\begin{align*}
    \mathcal{A}_U^\top(\lambda_U) = \frac{1}{2}({e_U} y_U^\top + {y_U} e_{U}^\top) + \alpha_U \Diag(e_U) + \sum_{(i, j) \in \textrm{CL}_U} \tau_U(i, j) E_{ij}, \\
    \mathcal{A}_V^\top(\lambda_V) = \frac{1}{2}({e_V} y_V^\top + {y_V} e_{V}^\top) + \alpha_V \Diag(e_V) + \sum_{(i, j) \in \textrm{CL}_V} \tau_V(i, j) E_{ij},
\end{align*}
\end{comment}
In particular, once the primal and dual SDPs have been solved approximately by a first-order method, the following theorem gives an upper bound on the optimal value of the SDP relaxation. The next proposition ensures the theoretical validity of the bounds produced within the branch-and-cut algorithm.
\begin{proposition}
\label{theorem:safe_bound}
Let ${Z^\star}$ be an optimal solution of \eqref{prob:primal} with objective function value $p^\star$. Consider the dual variables  ${y_U} \in \mathbb{R}^{\bar{n}}$, ${y_V} \in \mathbb{R}^{\bar{m}}$, $\alpha_U, \alpha_V \in \mathbb{R}$, ${t}_U \in \mathbb{R}^p$, ${t}_V \in \mathbb{R}^q$, $\tau_U \in \mathbb{R}^{|\overline{\textrm{CL}}_V|}$, $\tau_V \in \mathbb{R}^{|\overline{\textrm{CL}}_V|}$, ${Q} \in \mathbb{S}^{\bar{n}+\bar{m}}$, $Q \geq 0$. Let $\lambda_U = [y_U; \alpha_U; \tau_U]$ and $ \lambda_V = [y_V; \alpha_V; \tau_V]$. Set $\tilde{S} =\begin{bmatrix}
{\tilde{S}_{UU}} & {\tilde{S}_{UV}}\\
\tilde{S}_{UV}^\top & {\tilde{S}_{VV}}
\end{bmatrix}$, where ${\tilde{S}}_{UV} = -T_U A T_V^\top - {Q}_{UV}$, $\tilde{S}_{UU} = \mathcal{A}_U^\top({\lambda_U}) + \mathcal{B}_U^\top({t_U}) - Q_{UU}$ and $\tilde{S}_{VV} = \mathcal{A}_V^\top({\lambda_V}) + \mathcal{B}_V^\top({t_V}) - Q_{VV}$. Then, the following inequality holds:

\begin{align*}
    p^\star \leq {y_U^\top} 1_{\bar{n}} + y_V^\top 1_{\bar{m}} + k(\alpha_U + \alpha_V) - d_{\min} \sum_{i \colon  \lambda_i({\tilde{S}}) <0}\lambda_i({\tilde{S}}),
\end{align*}
where $d_{\min} = 1/\min_{j \in \{1, \dots, \bar{n}\}} {(e_U)_j} + 1/\min_{j \in \{1, \dots, \bar{m}\}}{(e_V)_j}$.
\end{proposition}

\vspace{-15pt}

\begin{proof}
    See Appendix \ref{appendix:safe}.
\end{proof}
In practice, this post-processing step yields an overestimate of the optimal value of the primal SDP \eqref{prob:primal}, which also serves as an upper bound for the original non-convex problem. Specifically, when the dual matrix ${\tilde{S}}$ is numerically positive semidefinite, the point $(\lambda_U, \lambda_V, t_U, t_V, {\tilde{S}}, {Q})$ satisfies the constraints of the dual problem \eqref{prob:dual}, and the expression ${y_U^\top} 1_{\bar{n}} + y_V^\top 1_{\bar{m}} + k(\alpha_U + \alpha_V)$ directly gives a valid upper bound. If ${\tilde{S}}$ has negative eigenvalues, a correction term $-d_{\min}\sum_{i \colon \lambda_i({\tilde{S}}) <0}\lambda_i({\tilde{S}})$ is added to ensure the bound is valid, as established by Proposition \ref{theorem:safe_bound}.

\subsection{Lower bound computation}\label{subsection:lower_bound}
In this section, we develop a rounding algorithm that recovers a feasible biclustering, and thus a lower bound on the optimal value of Problem \eqref{prob:or_constr_shr}, from the solution of the SDP relaxation solved at each node. After solving the SDP, we obtain a matrix $\tilde{Z}$, which is generally not feasible for the original low-rank problem since $\rank(\tilde{Z}) \neq k$. To extract a feasible set of bicliques from this solution, we employ the rounding procedure outlined in Algorithm \ref{alg:rounding_heuristic}.
Steps~1 and~2 follow a paradigm commonly adopted in constrained clustering problems, where an initial solution (typically obtained from $k$-means) is projected onto a feasible discrete set by solving an integer program that enforces the given constraints. Similar rounding procedures have been proposed, for instance, in the context of fair clustering~\citep{lawless2024fair} and cardinality-constrained clustering~\citep{rujeerapaiboon2019size}.
\begin{algorithm}[!ht]
\small
\caption{Rounding procedure for finding feasible bicliques from the SDP solution}
\label{alg:rounding_heuristic}
\textbf{Input}: {Data matrix $A$, number of biclusters $k$, must-link matrices $T_U$, $T_V$, projected cannot-link constraints $\overline{\textrm{CL}}_U$, $\overline{\textrm{CL}}_V$, solution $\tilde{Z}$ of the SDP relaxation.}
\begin{enumerate}
\item Cluster the rows of $\tilde{Z}_{UV}$ with $k$-means algorithm and record the row partition $\bar{X}_U$. Solve
{\small
\begin{align*}
 \max_{X_U \in \{0, 1\}^{\bar{n} \times k}}
\left\{ \inprod{X_U}{\bar{X}_U} : \ \begin{array}{l}
     X_U 1_k = 1_{\bar{n}}, \ X_U^\top 1_{\bar{n}} \geq 1_k, \\
    (X_U)_{ih} + (X_U)_{jh} \leq 1, \quad \forall h \in\{1, \dots, k\}, \ \forall(\mathcal{U}_i, \mathcal{U}_j) \in \overline{\textrm{CL}}_{U} 
    \end{array} \right\}
\end{align*}
}%
and let $\tilde{X}_U$ be the optimal solution. Set $\tilde{X}_U \leftarrow  T_U^\top \tilde{X}_U$.
\item Cluster the columns of $\tilde{Z}_{UV}$ with $k$-means algorithm and record the column partition $\tilde{X}_V$. Solve
{\small
\begin{align*}
\max_{X_V \in \{0, 1\}^{\bar{m} \times k}}
\left\{ \inprod{X_V}{\bar{X}_V}  : \ \begin{array}{l}
     X_V 1_k = 1_{\bar{m}}, \ X_V^\top 1_{\bar{m}} \geq 1_k, \\
    (X_V)_{ih} + (X_V)_{jh} \leq 1, \quad \forall h \in\{1, \dots, k\}, \ \forall(\mathcal{V}_i, \mathcal{V}_j) \in \overline{\textrm{CL}}_{V} 
    \end{array} \right\}
\end{align*}
}%
and let $\tilde{X}_V$ be the optimal solution. Set $\tilde{X}_V \leftarrow  T_V^\top \tilde{X}_V$.
\item Let $\tilde{W} = (\tilde{X}_U^\top \tilde{X}_U)^{-\frac{1}{2}}\tilde{X}_U^\top A \tilde{X}_V (\tilde{X}_V^\top \tilde{X}_V)^{-\frac{1}{2}}$. Solve the linear assignment problem
{\small
\begin{align*}
\centering
    P^\star = \argmax_{P \in \{0, 1\}^{k \times k}} \left\{ \inprod{P}{\tilde{W}} : \ P 1_k = 1_k, \ P^\top 1_k = 1_k \right\}.
\end{align*}}%
\item Permute the columns of $\tilde{X}_U$ according to $P^\star$ and set $U^\star_j \leftarrow \{u_i \in U \ : \ (\tilde{X}_U)_{ij} = 1\}$ for all $j \in \{1, \dots, k\}$.
\item Permute the columns of $\tilde{X}_V$ according to $P^\star$ and set $V^\star_j \leftarrow \{v_i \in V \ : \ (\tilde{X}_V)_{ij} = 1\}$ for all  $j \in \{1, \dots, k\}$.
\end{enumerate}
\textbf{Output}: {Feasible bicliques $\{(U^\star_1 \cup V^\star_1), \dots, (U^\star_k \cup V^\star_k) \}$.}
\end{algorithm}
Indeed, Steps 1 and 2 start by applying the $k$-means algorithm \citep{lloyd1982least} to the rows and columns of the submatrix $\tilde{Z}_{UV}$, respectively, yielding initial cluster assignments encoded in binary matrices $\bar{X}_U$ and $\bar{X}_V$. These partitions may violate the pairwise cannot-link constraints. To restore feasibility, each initial partition is refined by solving an integer linear program (ILP) whose objective maximizes the agreement with the reference $k$-means assignment while ensuring that every vertex is assigned to exactly one cluster, all clusters are non-empty, and no cannot-link constraint is violated. Must-link constraints do not need explicit inclusion in the ILPs, as they are enforced afterward by multiplying the ILP solutions by the corresponding must-link matrices ($T_U$ or $T_V$), which replicate the assignments for all linked vertices. Note that, if the set of pairwise constraints provided by the user is logically inconsistent, the infeasibility is automatically detected during the construction of feasible assignments in Steps 1 and/or 2. In Step 3, with feasible row and column clusters in hand, we construct a $k \times k$ weight matrix $\tilde{W}$, where each entry measures the density of the subgraph induced by each row–column cluster pair in $K_{n, m}$. Here, the goal is to assign each row cluster to exactly one column cluster in the best possible way. We then solve a linear assignment problem to find the permutation matrix $P^\star$ that maximizes the corresponding total density of the selected cluster pairings. This problem is solvable in polynomial time because the constraint matrix is totally unimodular. Finally, Steps 4 and 5 reorder the columns of $\tilde{X}_U$ and $\tilde{X}_V$ according to the optimal assignment $P^\star$, aligning row and column clusters to form a collection of $k$ feasible bicliques $\{(U^\star_j, V^\star_j)\}_{j=1}^k$ of $K_{n, m}$. Although the ILPs in Steps~1 and~2 are NP-hard in general, their favorable combinatorial structure---combined with the empirical observation that the reference partitions derived from the SDP solution are typically close to feasibility---allows them to be solved efficiently in practice using off-the-shelf ILP solvers. As shown in the experiments, their computational overhead is negligible compared to that of solving the SDP relaxation.

When the upper bound provided by the SDP relaxation is not sufficient to certify the optimality of the best feasible solution obtained via rounding, and no further violated inequalities can be identified (or the bound does not improve significantly after adding them), the generation of cutting planes is halted and the algorithm proceeds with branching. We generate a binary enumeration tree in which each node corresponds to a subproblem augmented with additional pairwise constraints. Specifically, at each branching node, a pair of vertices $(i,j)$ from either $\bar{U}$ or $\bar{V}$ is selected, and two child nodes are created by imposing, respectively, a must-link and a cannot-link constraint on that pair. This branching strategy clearly yields a finite branch-and-bound algorithm. The branching pair is selected according to the criterion in \cite{sudoso2024}, originally proposed for the unconstrained case. In any feasible rank-constrained solution, each pair $\mathcal{U}_i, \mathcal{U}_j \in \bar{U}$ satisfies $(Z_{UU})_{ij} \cdot ( (Z_{UU})_{ii} - (Z_{UU})_{ij}) = 0$ meaning that either $(Z_{UU})_{ij} = 0$ (different bicliques) or $(Z_{UU})_{ii} = (Z_{UU})_{ij}$ (same biclique). An analogous condition holds for $\mathcal{V}_i, \mathcal{V}_j \in \bar{V}$. The degree of violation of these conditions is used as a branching score. Specifically, for $(\mathcal{U}_i, \mathcal{U}_j) \in \bar{U} \times \bar{U}$, define 
$(b_U)_{ij} = \min \{ (Z_{UU})_{ij}, (Z_{UU})_{ii} - (Z_{UU})_{ij} \}$, and analogously for $(b_V)_{ij}$ on $\bar{V}$. Scores are scaled by $|\bar{U}|$ or $|\bar{V}|$ to ensure comparability, and the branching pair is chosen as $(i^\star, j^\star) \in \arg\max ( \{ |\bar{U}| \cdot (b_U)_{ij} \} \cup \{ |\bar{V}| \cdot (b_V)_{ij} \})$. In other words, we choose the pair whose assignment to the same or different bicliques is most ambiguous according to the SDP solution.

\section{Low-rank factorization heuristic}\label{section:low_rank}
As will be shown in the computational results, even when the SDP in the bounding routine is solved with state-of-the-art first-order methods, the main computational bottleneck remains the solution of the SDP relaxation itself. This is particularly critical for graphs arising in real-world biclustering applications. Indeed, solving large-scale SDPs is challenging due to the high dimensionality of the matrix variable and the computational cost of enforcing positive semidefiniteness. To address this, we propose a low-rank factorization heuristic based on the Burer--Monteiro (BM) approach~\cite{burer2003nonlinear, burer2005local}, which is particularly effective for SDPs with low-rank structure. The BM factorization replaces the matrix variable in the SDP with a product of smaller matrices, thereby reducing the problem's dimensionality and eliminating the explicit semidefinite constraint. Specifically, in Problem~\eqref{prob:sdp_shr}, we write the matrix \( Z \in \mathbb{S}^{\bar{n}+\bar{m}} \) as
\[
Z = \begin{bmatrix}
    Z_U \\
    Z_V
\end{bmatrix}
\begin{bmatrix}
    Z_U \\
    Z_V
\end{bmatrix}^\top
=
\begin{bmatrix}
Z_U Z_U^\top & Z_U Z_V^\top \\
Z_V Z_U^\top & Z_V Z_V^\top
\end{bmatrix},
\]
for some \( Z_U \in \mathbb{R}^{\bar{n} \times r} \) and \( Z_V \in \mathbb{R}^{\bar{m} \times r} \). This factorization implicitly enforces \( Z \succeq 0 \) and limits its rank to at most \( r \).
Using this representation, the SDP can be reformulated as the following non-convex problem over \( Z_U \in \mathbb{R}^{\bar{n} \times r} \) and \( Z_V \in \mathbb{R}^{\bar{m} \times r} \):
\begin{align}\label{prob:lr-sdp}
\max_{\substack{%
        Z \geq 0
      }} \ \left\{ \inprod{T_UAT_V^\top}{Z_U Z_V^\top} \ : \
\begin{array}{l}
     \inprod{\Diag(e_U)}{Z_UZ_U^\top} = k, \ Z_UZ_U^\top e_U = 1_{\bar{n}}, \\ (Z_UZ_U^\top)_{ij} = 0 \quad \forall (\mathcal{U}_i, \mathcal{U}_j) \in \overline{\textrm{CL}}_U,\\ 
     \inprod{\Diag(e_V)}{Z_VZ_V^\top} = k, \ Z_VZ_V^\top e_V = 1_{\bar{m}}, \\ (Z_VZ_V^\top)_{ij} = 0 \quad \forall (\mathcal{V}_i, \mathcal{V}_j) \in \overline{\textrm{CL}}_V.
\end{array} \right\}.
\end{align}
Problems \eqref{prob:sdp_shr} and \eqref{prob:lr-sdp} are equivalent for any value $r \geq r_\textrm{min}$, where $r_\textrm{min}$ is the smallest
among the ranks of all optimal solutions of Problem \eqref{prob:sdp_shr}. %This equivalence is stated in the sense that it is possible to construct an optimal solution of one problem from a global solution of the other and vice versa. 
The advantage of Problem \eqref{prob:lr-sdp} compared to \eqref{prob:sdp_shr} is that its matrix variable $Z_U$ (resp. $Z_V$) has significantly less entries than that of Problem \eqref{prob:sdp_shr} when $r \ll \bar{n}$ (resp. $r \ll \bar{m}$). 
However, since Problem \eqref{prob:lr-sdp} is non-convex, it may have stationary points that are not globally optimal. Typically, as soon as $r$ is slightly larger than $r_\textrm{min}$, local optimization algorithms seem to solve the factorized problem globally. This empirical behavior is observed in \cite{burer2005local, boumal2020deterministic, cifuentes2021burer}.

To efficiently compute stationary points of Problem \eqref{prob:lr-sdp}, we employ the augmented Lagrangian method (ALM) \cite{bertsekas1997nonlinear}. We first replace the constraint $Z \geq 0$ (i.e., $Z_U Z_U^\top \geq 0$, $Z_UZ_V^\top \geq 0$, and $Z_VZ_V^\top \geq 0$) with the stronger constraints $Z_U \geq 0$ and $Z_V \geq 0$ that are easier to enforce in the augmented Lagrangian setting. Moreover, without loss of generality, we bound all entries of \( Z_U \) and \( Z_V \) by the trivial upper bound 1, defining $\Omega_U = \{ Z_U \in \mathbb{R}^{\bar{n} \times r} : 0 \leq (Z_U)_{ij} \leq 1\}$ and $\Omega_V = \{ Z_V \in \mathbb{R}^{\bar{m} \times r} : 0 \leq (Z_V)_{ij} \leq 1\}$, where projections onto these sets are available in closed form. Then, Problem \eqref{prob:lr-sdp} can be converted into an equality-constrained problem over $\Omega_U \times \Omega_V$ and expressed as
\begin{align}\label{prob:bm_standard}
\min_{Z_U \in \Omega_U, Z_V \in \Omega_V} \ \left\{ -\inprod{T_UAT_V^\top}{Z_U Z_V^\top} \ : \
\begin{array}{l}
     \mathcal{A}_U(Z_U Z_U^\top) = b_U, \\ 
     \mathcal{A}_V(Z_V Z_V^\top) = b_V
    \end{array} \right\}.
\end{align}

Let $\beta > 0$ be a given penalty parameter, $\lambda_U \in \mathbb{R}^{\bar{n} + |\overline{\mathrm{CL}}_U| + 1}$ and $\lambda_V \in \mathbb{R}^{\bar{m} + |\overline{\mathrm{CL}}_V| + 1}$ be the Lagrange multipliers associated with the constraints $\mathcal{A}_U(Z_U Z_U^\top) = b_U$ and $\mathcal{A}_V(Z_V Z_V^\top) = b_V$, respectively. The augmented Lagrangian function associated with Problem \eqref{prob:bm_standard} is defined by
{
\begin{align*}
    {L}_\beta(Z_U, Z_V, \lambda_U, \lambda_V) &= -\inprod{T_UAT_V^\top}{Z_U Z_V^\top} \\
    &+ \lambda_U^\top ( \mathcal{A}_U(Z_U Z_U^\top)-b_U ) + \lambda_V^\top (\mathcal{A}_V(Z_VZ_V^\top) - b_V) \\
    &+ \frac{\beta}{2} \| \mathcal{A}_U(Z_UZ_U^\top) - b_U\|^2 + \frac{\beta}{2} \| \mathcal{A}_V(Z_VZ_V^\top) - b_V\|^2.
\end{align*}}%
The ALM used to solve Problem~\eqref{prob:bm_standard} is outlined in Algorithm~\ref{alg:ALM}. This iterative procedure alternates between solving the augmented Lagrangian subproblem (Step~1), updating the Lagrange multipliers (Step~2), checking for convergence (Step~3), and updating the penalty parameter (Step~4). At each outer iteration \(k\), the algorithm computes a solution of the augmented Lagrangian subproblem in \eqref{prob:alm_subproblem} using the current values of the multipliers \(\lambda_U^k\), \(\lambda_V^k\), and the penalty parameter \(\beta_k\). This subproblem is not solved exactly; instead, the solution is required to satisfy a projected gradient condition with tolerance $\varepsilon_k$, ensuring approximate first-order stationarity. Moreover, the fulfillment of \eqref{eq:norm_sub_U} and \eqref{eq:norm_sub_V} means that the projected gradient of the Lagrangian with multipliers $\lambda_U^{k+1}$ and $\lambda_V^{k+1}$ approximately vanishes with precision $\varepsilon_k$. Once the subproblem is solved, the residuals of the constraints are evaluated, and the Lagrange multipliers are updated accordingly. To guide the algorithm toward feasibility, the penalty parameter is adjusted based on the observed decrease in constraint violation: if the reduction is sufficient, the penalty is left unchanged; otherwise, it is increased by a factor \(\gamma > 1\). The parameter \(\tau \in (0,1)\) controls the required improvement. The process continues until a suitable stopping criterion is satisfied, yielding a
point $(Z_U^k, Z_V^k)$ with its associated Lagrange multipliers $\lambda_U^{k+1}$ and $\lambda_V^{k+1}$ satisfies the Karush-Kuhn-Tucker (KKT) conditions, as given in the definition below.

\begin{definition}\label{def:ekkt}
We say that $(\bar{Z}_U, \bar{Z}_V) \in \mathbb{R}^{\bar{n} \times r} \times \mathbb{R}^{\bar{m} \times r}$ is a KKT point for Problem \eqref{prob:bm_standard} if there exist $\lambda_U \in \mathbb{R}^{\bar{n}+\left|\overline{\textrm{CL}}_U\right| + 1}$ and $\lambda_V \in \mathbb{R}^{\bar{m}+\left|\overline{\textrm{CL}}_V\right| + 1}$ such that
\begin{align*}
    \mathcal{A}_U(\bar{Z}_U \bar{Z}_U) &= b_U, \ \bar{Z}_U = \Pi_{\Omega_U} ( \bar{Z}_U - (2\mathcal{A}_U^\top(\lambda_U) \bar{Z}_U - T_U A T_V^\top \bar{Z}_V)), \\
    \mathcal{A}_V(\bar{Z}_V \bar{Z}_V) &= b_V, \ 
    \bar{Z}_V = \Pi_{\Omega_V} ( \bar{Z}_V -( 2\mathcal{A}_V^\top(\lambda_V) \bar{Z}_V - T_V A^\top T_U^\top \bar{Z}_U)).
\end{align*}
\end{definition}

%For a comprehensive discussion on the augmented Lagrangian method applied to non-convex optimization problems and beyond, see \cite{andreani2008augmented, andreani2008augmented2, andreani2011sequential, birgin2014practical, birgin2020complexity}.

%$r_V^{k-1} = \mathcal{A}_V(Z_V^{k-1} (Z_V^{k-1})^\top) - b_V$
%such that $\lim_{k\to\infty} \varepsilon_k = 0$
\begin{algorithm}[!ht]
\small
\caption{Augmented Lagrangian method for Problem \eqref{prob:bm_standard}}
\label{alg:ALM}
\begin{algorithmic}
\Require Initial point $(Z_U^0, Z_V^0) \in \Omega_U \times \Omega_V$, $\beta_1 > 0, \tau \in (0,1)$, $\gamma > 1$, sequence $\{\varepsilon_k\} \to 0$, initial multipliers $\lambda_U^0 = 0, \lambda_V^0 = 0$.

\For{$k = 1, 2, \dots$}

        \Comment{Step 1: Solve the subproblem}
        \State Find an approximate solution $(Z_U^{k}, Y_V^{k})$ of 
        \begin{align}\label{prob:alm_subproblem}
            \min_{(Z_U, Z_V) \in \Omega_U \times \Omega_V} \ L_{\beta_{k}}(Z_U, Z_V, {\lambda}_U^{k}, {\lambda}_V^{k})
        \end{align}
\State satisfying
\begin{align}
\label{eq:norm_sub_U}
      \|\Pi_{\Omega_U}(Z_U^k-\nabla_{Z_U}L_{\beta_k}
                         (Z_U^k, Z^k_V, {\lambda}_U^k, {\lambda}_V^k)) - Z_U^k \| &\leq \varepsilon_k,\\
\label{eq:norm_sub_V}
      \|\Pi_{\Omega_V}(Z_V^k-\nabla_{Z_V}L_{\beta_k}
                         (Z_U^k, Z^k_V, {\lambda}_U^k,{\lambda}_V^k)) - Z_V^k \| &\leq \varepsilon_k;
\end{align}

        \Comment{Step 2: Estimate multipliers}

        Define $r_U^{k} = \mathcal{A}_U(Z_U^{k} (Z_U^{k})^\top) - b_U$, \ $r_V^{k} = \mathcal{A}_V(Z_V^{k} (Z_V^{k})^\top) - b_V$;

        \State Set $\lambda_U^{k+1} = \lambda_U^{k} + \beta_k r_U^k$, \ $\lambda_V^{k+1} = \lambda_V^{k} + \beta_k r_V^k$;

        \Comment{Step 3: Check convergence}
        \If{stopping criterion is satisfied}
        %\If{$\max\{\|
        %r_U^k\|, \| r_V^k\|\} = 0$}
        \State Set $\bar{Z}_U = Z_U^{k}$, $\bar{Z}_V = Z_V^{k}$ and \textbf{break};
        \EndIf

        \Comment{Step 4: Update the penalty}
        \If{$\max\{\|r_U^{k}\|, \|r_V^k \| \} \leq \tau \max\{\|r_U^{k-1} \|, \|r_V^{k-1}\| \}$}
        \State $\beta_{k+1} = \beta_k$;
        \Else{}
        \State $\beta_{k+1} = \gamma \beta_k$;
        \EndIf

\EndFor
\Ensure KKT point ${(\bar{Z}_U, \bar{Z}_V)}$.
\end{algorithmic}
\end{algorithm}

The convergence theory of the ALM is based on the definition of ``approximate-KKT point sequence'' as in \cite{andreani2011sequential, birgin2014practical}, which is given in the proposition below.

\begin{comment}
Like all penalty-type methods, augmented Lagrangian methods suffer from the drawback that they generate accumulation points which are not necessarily feasible for the given optimization problem. The following (standard) result therefore presents some conditions under which it is guaranteed that limit points are feasible. From [book], Theorem 6.5, we can deduce that each limit point of the sequence of iterates is feasible as expressed in the following proposition. Theorem 4.1 in \cite{andreani2008augmented}
\begin{proposition}
    Let $\{(Z_U^k, Z_V^k)\}$ be the sequence generated by Algorithm \ref{alg:ALM}. Then, if the sequence of penalty parameters $\{\beta_k\}$ is bounded, every limit point $(\bar{Z}_U, \bar{Z}_V)$ of $\{(Z_U^k, Z_V^k)\}$ is a feasible point of Problem \eqref{prob:bm_standard}.
\end{proposition}
Note that, even in the case where a limit point is not necessarily
feasible, it still contains some useful information in the sense that it is at least a stationary point for the constraint violation, i.e., an optimization problem that minimizes the sum of squares of upper-level infeasibility subject to lower-level feasibility \cite{andreani2008augmented}.
Next, we show that feasible limit points are also KKT points of Problem \eqref{prob:bm_standard}.
\end{comment}

\begin{proposition}\label{theorem:sequential}
Let $\{(Z_U^k, Z_V^k)\}$ be the sequence generated by Algorithm \ref{alg:ALM}. Assume that $K \subseteq \{0, 1, \dots\}$ is an infinite subsequence of indices such that $\lim_{k \in K} (Z_U^k, Z_V^k) = (\bar{Z}_U, \bar{Z}_V)$ and the sequence $\{\varepsilon_k\}$ is such that $\lim_{k \in K}\varepsilon_k = 0$. Then, every feasible limit point of the sequence $\{(Z_U^k, Z_V^k)\}$ satisfies the approximate-KKT (AKKT) condition for Problem \eqref{prob:bm_standard} given by
    \begin{align*}
        \lim_{k \in K} \|\mathcal{A}_U(Z_U^k (Z_U^k)^{\top}) - b_U \| & = 0, \\
        \lim_{k \in K} \|\mathcal{A}_V(Z_V^k (Z_V^k)^{\top}) - b_V \| & = 0, \\
        \lim_{k \in K} \| \Pi_{\Omega_U}(Z_U^k - (2\mathcal{A}_U^\top(\lambda_U^{k+1}) Z_U^k - T_U A T_V^\top Z_V^k)) - Z_U^k \| &= 0, \\
        \lim_{k \in K} \| \Pi_{\Omega_V}(Z_V^k - (2\mathcal{A}_V^\top(\lambda_V^{k+1}) Z_V^k - T_V A^\top T_U^\top Z_U^k)) - Z_V^k \| &= 0.
    \end{align*}
\end{proposition}

\vspace{-15pt}

\begin{proof}
See Appendix \ref{appendix:alm_conv_proof}.
\end{proof}

Note that, AKKT is a “sequential” optimality condition in the sense that (standard) KKT is a “pointwise” condition. 
It is known that the AKKT condition implies the usual KKT condition under constraint qualifications \cite{andreani2011sequential}.
In particular, since LICQ (linear independence of the gradients of active constraints) holds at any feasible point of Problem \eqref{prob:bm_standard}, then we have that if $(\bar{Z}_U, \bar{Z}_V)$ satisfies the AKKT conditions, then $(\bar{Z}_U, \bar{Z}_V)$ is a KKT point of this problem. Further theoretical properties of the adopted augmented Lagrangian framework are discussed in \cite{birgin2014practical, birgin2020complexity}.

%Note that, as stated in \cite{birgin2020complexity} even if the sequence $\{\varepsilon_k\}$ does not tend to zero, Algorithm () finds $\varepsilon$-KKT points of the infeasibility measure, i.e., an optimization problem where $\|\mathcal{A}_U(Z_UZ_U^\top) - b_U \|_F^2 + \| \mathcal{A}_U(Z_VZ_V^\top) - b_V \|_F^2$ is minimized subject to $Z_U \in \Omega_U$ and $Z_V \in \Omega_V$ and that, when $\{\varepsilon_k\}$ tends to zero, feasible limit points satisfy a sequential optimality condition, as stated in Theorem \ref{theorem:sequential}.

Proposition \ref{theorem:sequential} provides a practical stopping criterion for declaring convergence of Algorithm \ref{alg:ALM}. Thus, we stop the algorithm when an iterate $(Z_U^k, Z_V^k)$ with its associated Lagrange multipliers $\lambda_U^{k+1}$ and $\lambda_V^{k+1}$ satisfies the condition
\begin{align}\label{eq:alm_tol_eps}
\max \left\{
\begin{aligned}
&\|\mathcal{A}_U(Z_U^k (Z_U^k)^{\top}) - b_U \|, \\
&\|\mathcal{A}_V(Z_V^k (Z_V^k)^{\top}) - b_V \|, \\
&\| \Pi_{\Omega_U}(Z_U^k - (2\mathcal{A}_U^\top(\lambda_U^{k+1}) Z_U^k - T_U A T_V^\top Z_V^k)) - Z_U^k \|, \\
&\| \Pi_{\Omega_V}(Z_V^k - (2\mathcal{A}_V^\top(\lambda_V^{k+1}) Z_V^k - T_V A^\top T_U^\top Z_U^k)) - Z_V^k \|
\end{aligned}
\right\}
\leq \varepsilon_{\textrm{ALM}},
\end{align}
where $\varepsilon_{\text{ALM}} > 0$ is a small tolerance.
%A sufficient condition for the existence of at least one limit point (i.e., a convergent subsequence, i.e.,there exists an infinite subset $K \subseteq \{0, 1, \dots\}$ such that $\lim_{k \in K, \ k\to\infty} \mathcal{\bar{L}}(Y_U^k, Y_V^k) = (\bar{Y}_U, \bar{Y}_V)$) is the boundedness of the feasible region. However, the feasible region of Problem \eqref{prob:alm_subproblem} is unbounded due to the constraints $Y_U \geq 0$ and $Y_V \geq 0$. To enforce boundedness, one can introduce valid upper bounds on all variables by adding artificial constraints to $\Omega_U$ and $\Omega_V$. Since this modification leads to a bound-constrained problem, Algorithm \ref{alg:ALM} remains applicable, and Theorem \ref{theorem:convergence} holds.

Upon convergence, the algorithm returns a point $(\bar{Z}_U, \bar{Z}_V)$, which is used to construct the matrix $Z_{\mathrm{ALM}} = [\bar{Z}_U; \bar{Z}_V][\bar{Z}_U; \bar{Z}_V]^\top$. This matrix serves as an approximate solution to the SDP relaxation in \eqref{prob:sdp_shr} and is subsequently used to initialize the rounding procedure described in Algorithm~\ref{alg:rounding_heuristic}, ultimately producing a feasible solution to Problem~\eqref{prob:or_constr_shr}. Note that, due to the nature of Algorithm~\ref{alg:ALM}, its output is generally dependent on the choice of the starting point. To reduce the risk of converging to a poor-quality solution, the procedure is executed in a multi-start fashion, using different randomly generated initializations and retaining the one with the best objective function value.

The main challenge of the ALM lies in solving the subproblem \eqref{prob:alm_subproblem} at each iteration. 
In practice, however, it suffices to compute an approximate stationary point that satisfies the first-order optimality conditions in \eqref{eq:norm_sub_U} and \eqref{eq:norm_sub_V}. To efficiently handle the complexity of the subproblem, we exploit its block structure in the variables $(Z_U, Z_V)$. {In fact, the separability of the constraints and the form of the augmented Lagrangian naturally lend themselves to a decomposition approach.}

For fixed $\lambda_U^k, \lambda_V^k$ and $\beta_k$, we denote the augmented Lagrangian function as $\bar{L}(Z_U, Z_V) := L_{\beta_k}(Z_U, Z_V, \lambda_U^k, \lambda_V^k)$. Then, we adopt a two-block Gauss-Seidel scheme that alternates between optimizing over $Z_U$ and $Z_V$, treating one block as fixed while updating the other. This results in the following minimization procedure at the $(t+1)$-th iteration:
\begin{align}
    Z_U^{t+1} &\approx \argmin_{Z_U \in \Omega_U} \ \bar{L}(Z_U, Z_V^t), \label{eq:decomp_U} \\
    Z_V^{t+1} &\approx \argmin_{Z_V \in \Omega_V} \ \bar{L}(Z_U^{t+1}, Z_V), \label{eq:decomp_V}
\end{align}
where each of these subproblems is a box-constrained optimization problem of size smaller than the one in \eqref{prob:alm_subproblem}. 

The overall algorithm, described in Algorithm~\ref{alg:GS}, is an inexact block coordinate descent scheme. Each block is updated using projected gradient descent with Armijo line search, avoiding exact minimization and thus reducing the computational effort. This inexact approach fits naturally within the outer ALM loop and its convergence to an approximate stationary point under mild assumptions follows from \cite{cassioli2013convergence, galli2020unified}.

%At each outer iteration $t$, the algorithm first updates $Z_U$ while keeping $Z_V$ fixed at $Z_V^t$. This update is performed using projected gradient descent: a steplength $\alpha_k$ is selected, and the descent direction $D_U^k$ is computed as the projection of a gradient step onto the feasible set $\Omega_U$. A backtracking line search with parameters $\theta$ and $\sigma$ is used to find a $\lambda_k$ satisfying the Armijo-type condition of sufficient decrease of the objective function. If the norm $\|D_U^k\|$ of the direction falls below a prescribed tolerance $\varepsilon$, the update is accepted and the inner loop terminates with $Z_U^{t+1}$. The algorithm then proceeds to update $Z_V$, now keeping $Z_U$ fixed at its newly updated value $Z_U^{t+1}$. The same projected gradient descent and backtracking line search procedure is applied to obtain $Z_V^{t+1}$. At the end of each outer iteration, convergence is checked by evaluating the maximum of the changes in the two blocks. If the norm of two consecutive iterates for both block falls below the tolerance $\varepsilon$, then the algorithm terminates, and the current point $(\bar{Z}_U, \bar{Z}_V)$ is returned as an approximate stationary point of Problem \eqref{prob:alm_subproblem}.

A crucial issue for the convergence rate of the projected gradient method is the choice of the steplength $\alpha_k$. Following the original ideas of Barzilai and Borwein \cite{barzilai1988two}, several steplength updating strategies have been proposed to accelerate the slow convergence exhibited by gradient methods. The steplength parameter $\alpha_k$ is computed by an adaptive alternation of the Barzilai–Borwein (BB) rule as in \cite{zhou2006gradient, frassoldati2008new}, whose explicit expressions are given as follows. Consider an inner iteration $k$ an outer iteration $t$ for the computation of $Z_U^{t+1}$ starting from $(Z_U^k, Z_V^t)$ and define the quantities
\begin{align*}
    \alpha_k^{\textrm{BB1}} = \frac{(s_U^{k-1})^\top s_U^{k-1}}{(s_U^{k-1})^\top z_U^{k-1}}, \quad \alpha_k^{\textrm{BB2}} = \frac{(s_U^{k-1})^\top z_U^{k-1}}{(z_U^{k-1})^\top z_U^{k-1}},
\end{align*}
where, $s_U^{k-1} = \textrm{vec}(Z_U^k - Z_U^{k-1})$, $z_U^{k-1}=\textrm{vec}(\nabla_{Z_U}{\bar{L}}(Z_U^k, Z_U^t) - \nabla_{Z_V}{\bar{L}}(Z_U^{k-1}, Z_U^t))$ and $\textrm{vec}(\cdot)$ denotes the vectorization operator. The same quantities are computed with respect to $Z_V$ for the computation of $Z_V^{t+1}$ starting from $(Z_U^{t+1}, Z_V^k)$. Hence, we use the following adaptive alternation of the BB rule
\begin{align}\label{eq:alpha_selection}
    \alpha_k =
\begin{cases}
1 & \text{if } k = 0, \\
\min \left\{ \alpha_j^{\text{BB2}} \mid j = \max\{1, k - M_\alpha\}, \ldots, k \right\} & \text{if } \alpha_k^{\text{BB2}} / \alpha_k^{\text{BB1}} < \tau_{\alpha}, \\
\alpha_k^{\text{BB1}} & \text{otherwise},
\end{cases}
\end{align}
where $M_\alpha$ is a fixed nonnegative integer and $\tau \in (0, 1)$.
The convergence result for Algorithm \ref{alg:GS} is given below in Proposition \ref{proposition:alm_sub_conv}.

\begin{algorithm}[!ht]
\small
\caption{Gauss-Seidel decomposition with projected gradient descent for the augmented Lagrangian suproblem}
\label{alg:GS}
\begin{algorithmic}
\Require Augmented Lagrangian function ${\bar{L}}$, starting point $(Z_U^0, Z_V^0) \in \Omega_U \times \Omega_V$, parameters $\sigma, \theta, \tau_{\alpha} \in (0,1)$, nonnegative integer $M_\alpha$ and small tolerance $\varepsilon > 0$.

\For{$t = 0, 1, 2, \dots$}

    \Comment{Problem \eqref{eq:decomp_U}: Update $Z_U$ with fixed $Z_V$}
    \For{$k = 0,1,2,\dots$}
        \State Choose steplength $\alpha_k$ using Eq. \eqref{eq:alpha_selection} and set $\lambda_k = 1$;
        \State Compute direction $D_U^{k} = \Pi_{\Omega_U}(Z_U^{k} - \alpha_k \nabla_{Z_U} \bar{{L}}(Z_U^k, Z_V^t)) - Y_U^{k}$;
        \If{$\|D_U^{k}\| \leq \varepsilon$}
            \State Set $Z_U^{t+1} = Z_U^{k}$ and \textbf{break};
        \EndIf
        \While{$\bar{{L}}(Z_U^{k} + \lambda_k D_U^k, Z_V^t) > \bar{{L}}(Z_U^{k}, Z_V^{t}) + \sigma \lambda_k \inprod{\nabla_{Z_U} \bar{{L}}({Z}_U^{k}, Z_V^{t})}{D_U^{k}}$}
            \State Set $\lambda_k = \theta \lambda_k$;
        \EndWhile
        \State Update $Z_U^{k+1} = Z_U^{k} + \lambda_k D_U^{k}$;
    \EndFor
    
    \Comment{Problem \eqref{eq:decomp_V}: Update $Z_V$ with fixed $Z_U$}
    \For{$k = 0,1,2,\dots$}
        \State Choose steplength $\alpha_k$ using Eq. \eqref{eq:alpha_selection} and set $\lambda_k = 1$;
        \State Compute direction $D_V^{k} = \Pi_{\Omega_V}(Z_V^{k} - \alpha_k \nabla_{Z_V} \bar{{L}}(Z_U^{t+1}, Z_V^{k})) - Z_V^{k}$;
        \If{$\|D_V^{k}\| \leq \varepsilon$}
            \State Set $Z_V^{t+1} = Z_V^{k}$ and \textbf{break};
        \EndIf
        \While{$\bar{{L}}(Z_U^{t+1}, Y_V^{k} + \lambda_k D_V^{k}) > \bar{{L}}(Z_U^{t+1}, Z_V^{k}) + \sigma \lambda_k \inprod{\nabla_{Y_V} \bar{{L}}({Z}_U^{t+1}, Z_V^{k})}{D_V^{k}}$}
            \State Set $\lambda_k = \theta \lambda_k$;
        \EndWhile
        \State Update $Z_V^{k+1} = Z_V^{k} + \lambda_k D_V^{k}$;
    \EndFor
    
    \Comment{Check convergence}
    \If{$\max\left\{\|Z_U^{t+1} - Z_U^{t} \|, \|Z_V^{t+1} - Z_V^{t}\|\right\} \leq \varepsilon$}
            \State Set $\bar{Z}_U = Z_U^{t+1}$ and $\bar{Z}_V = Z_V^{t+1}$ and \textbf{break};
    \EndIf

\EndFor
\Ensure Approximate stationary point $(\bar{Z}_U, \bar{Z}_V)$.
\end{algorithmic}
\end{algorithm}

\begin{proposition}\label{proposition:alm_sub_conv}
Given a small tolerance $\varepsilon > 0$, the sequence $\{(Z_U^t, Z_V^t)\}$ generated by Algorithm~\ref{alg:GS} admits limit points, and every limit point $(\bar{Z}_U, \bar{Z}_V)$ is an approximate stationary point of the augmented Lagrangian subproblem $\min_{Z_U \in \Omega_U, Z_V \in \Omega_V} \bar{L}(Z_U, Z_V)$.
\end{proposition}

\vspace{-15pt}

\begin{proof}
See Appendix \ref{appendix:alm_sub_conv_proof}.
\end{proof}

\section{Computational results}\label{section:computational_results}
In this section, we report the computational results. We first compare the exact solver against Gurobi on small-scale random graphs. Next, we evaluate the exact and the heuristic methods on medium-scale real-world instances under different sets of constraints. Finally, we test the heuristic on large-scale instances and we validate its effectiveness by using machine learning metrics.

\subsection{Implementation details}
The exact solver, named {\tt CBICL-BB} (Constrained Biclustering/Biclique Branch-and-Bound), is implemented in C++ with some routines written in MATLAB (version 2022b). The SDP relaxation at each node is solved by means of SDPNAL+, a state-of-the-art MATLAB software that implements an augmented Lagrangian method for SDPs with bound constraints \citep{sun2020sdpnal+}. We set the accuracy tolerance of SDPNAL+ to $10^{-4}$ in the relative KKT residual. We use Gurobi (version 11.0.0) to solve the ILPs within the rounding heuristic. 
As for the cutting-plane setting, we randomly separate at most $100,000$ valid cuts (i.e., pair and triangular inequalities), sort them in decreasing order with respect to the violation, and add the first $10,000$ violated ones in each iteration. Although the pair inequalities are theoretically implied by the triangle inequalities together with the row-sum equalities (see Remark~7 in~\cite{de2022ratio}), in practice, violated pair inequalities are added iteratively, as their inclusion has been observed to enhance the numerical stability of {SDPNAL+}, leading to faster convergence and smaller primal–dual residuals. In addition, inactive pair and triangle inequalities are removed at each iteration using a tolerance of $10^{-4}$.
We stop the cutting-plane procedure when the separation routine does not find violated inequalities or when the relative difference of the upper bound between two consecutive iterations is less than or equal to $10^{-3}$. Moreover, the obtained set of inequalities at every node are inherited by the child nodes when branching is performed. Finally, we visit the tree with the best-first search strategy and require an optimality tolerance of $10^{-3}$ on the gap, i.e., we terminate the branch-and-cut method when {$(UB - LB)/UB \leq 10^{-3}$}, where UB and LB denote the best upper and lower bounds, respectively. 

%\max\{\varepsilon_{\textrm{ALM}}, 0.5^k\}
The heuristic solver is named \texttt{CBICL-LR} (Constrained Biclustering/Biclique {Low-Rank}) and is implemented in MATLAB. In line with standard practice (see, e.g., \cite{cifuentes2021burer}), we select \( r \) as the smallest integer satisfying \( r(r+1)/2 > c \), where \( c \) is the number of constraints in the original SDP. 
In Algorithm \ref{alg:ALM}, the point $(Z_U^0, Z_V^0)$ is randomly choosen within the set $\Omega_U \times \Omega_V$ and parameters $\beta_1 = 10$, $\gamma = 2$, $\tau = 0.5$. The ALM stops when the overall residual in Eq. \eqref{eq:alm_tol_eps} falls below~$\varepsilon_{\textrm{ALM}} = 10^{-3}$. The inner loop terminates once the block-wise subproblem attains an approximate stationary point with tolerance $\varepsilon_k = \varepsilon_{\textrm{ALM}}$ for all $k \in \mathbb{N}$. The adopted parameters for the Armijo backtracking line-search in Algorithm \ref{alg:GS} are $\theta = 0.5$, $\sigma = 10^{-4}$ and the parameters for 
the steplength selection are $M_\alpha = 2$ and $\tau_{\alpha} = 0.1$.
All experiments were run on a laptop equipped with an \mbox{Intel\,(R)\,i7-13700H} CPU, $32$\,GB\,RAM, and Ubuntu~22.04.3~LTS. To promote reproducibility and facilitate further research in the field, the source code is publicly available at \url{https://github.com/antoniosudoso/cbicl}.

\subsection{Experiments on artificial instances}
As previously discussed, general-purpose solvers implement spatial branch-and-bound algorithms designed to address non-convex QCQPs, such as Problem \eqref{prob:or_constr_shr}.
Thus, as the first set of experiments, we compare the proposed {\tt CBICL-BB} against {\tt GUROBI} on small-scale artificial graphs. We generate these graphs according to the so-called ``planted biclustering model'' \cite{ames2014guaranteed}. That is, we construct a weighted complete bipartite graph $K_{n, m}$ and we plant a set of $k$ disjoint bicliques $\{(U_1 \cup U_1), \dots, (U_k \cup V_k)\}$. We build the corresponding weight matrix $A$ such that if $u_i \in U_h$, $v_j \in V_h$ for some $h \in \{1, \dots, k\}$, then $A_{ij}$ is sampled independently from the uniform distribution over the interval $[0, 1]$. If $u_i$ and $v_j$ belong to different bicliques of $K_{n, m}$, then $A_{ij} = 0$. Finally, we corrupt the entries of the data matrix $A$ by adding Gaussian noise with mean 0 and standard deviation 0.25. We consider random graphs with varying numbers of vertices $(n, m) \in \{10, 15, 20, 25\}$ with $n = m$, numbers of bicliques $k \in \{2, 3\}$, and for each constraint configuration $(|\textrm{ML}_U|, |\textrm{CL}_U|, |\textrm{ML}_V|, |\textrm{CL}_V|)$ we generate three random sets of constraints with different seeds. For each graph, we create several instances differing in the type and amount of pairwise constraints. To generate such constraint sets we randomly choose pairs of vertices and set a constraint between them depending on whether they belong to the same row/column cluster (must-link constraint) or different (cannot-link constraint). We evaluate the effect of pairwise constraints applied, individually and jointly, to the two ends of the bipartite graph. Specifically, we consider three configurations: must-link and cannot-link constraints are enforced on the vertices in set $U$ only $(n/4, n/4, 0, 0)$, $(n/2, n/2, 0, 0)$; constraints are applied only on the vertices in the set $V$ $(0, 0, m/4, m/4)$, $(0, 0, m/2, m/2)$; both vertex sets are subject to pairwise constraints $(n/4, n/4, m/4, m/4)$, $(n/2, n/2, m/2, m/2)$. All values are rounded to the nearest integer using half-up rounding. Overall, the test set consists of 8 artificial graphs, labeled using the notation $\{n\}\_\{m\}\_\{k\}$, each evaluated under 6 constraint configurations and 3 random seeds per configuration. This results in a total of 144 constrained biclustering instances.

%Then, for each set, we consider 3 combinations of pairwise constraints: only ML and CL constraints, only ML and CL constraints, or an equal number of ML and CL constraints.

\begin{table}[!ht]
\centering
\footnotesize
\caption{Comparison between {\tt GUROBI} and {\tt CBICL-BB} on small-scale artificial graphs. We report the computational time, optimality gap, and number of nodes explored. Results are averaged over three sets of constraints for each configuration.}
%In column ``Avg. Gap (\%)'', the symbol ``$< \ell$'' indicates that all instances for the given constraint configuration were solved within the required optimality tolerance.}
\label{table:gurobi}
\begin{tabular}{l l r r r r r r}
\toprule
\multirow{2}{*}{Graph} & \multirow{2}{*}{Constr.} &
\multicolumn{3}{c}{{\tt GUROBI}} & \multicolumn{3}{c}{{\tt CBICL-BB}} \\
\cmidrule(lr){3-5} \cmidrule(lr){6-8}
& & \multicolumn{3}{c}{Avg.} & \multicolumn{3}{c}{Avg.} \\
& & Time \tiny{(s)} & Gap \tiny{(\%)} & Nodes & Time \tiny{(s)} & Gap \tiny{(\%)} & Nodes\\
\midrule
\texttt{10\_10\_2} & (3, 3, 0, 0) & 1.05 & $< \ell$ & 421.67 & 1.97 & $< \ell$ & 1.00 \\
\texttt{10\_10\_2} & (5, 5, 0, 0) & 0.75 & $< \ell$ & 793.00 & 1.90 & $< \ell$ & 1.00 \\
\texttt{10\_10\_2} & (0, 0, 3, 3) & 1.44 & $< \ell$ & 862.67 & 1.77 & $< \ell$ & 1.00 \\
\texttt{10\_10\_2} & (0, 0, 5, 5) & 0.72 & $< \ell$ & 537.00 & 1.70 & $< \ell$ & 1.00 \\
\texttt{10\_10\_2} & (3, 3, 3, 3) & 0.60 & $< \ell$ & 268.33 & 1.60 & $< \ell$ & 1.00 \\
\texttt{10\_10\_2} & (5, 5, 5, 5) & 0.18 & $< \ell$ & 1.00 & 1.47 & $< \ell$ & 1.00 \\
\texttt{10\_10\_3} & (3, 3, 0, 0) & 30.50 & $< \ell$ & 4942.00 & 1.20 & $< \ell$ & 1.00 \\
\texttt{10\_10\_3} & (5, 5, 0, 0) & 8.94 & $< \ell$ & 9146.33 & 1.50 & $< \ell$ & 1.00 \\
\texttt{10\_10\_3} & (0, 0, 3, 3) & 35.43 & $< \ell$ & 6222.00 & 1.67 & $< \ell$ & 1.00 \\
\texttt{10\_10\_3} & (0, 0, 5, 5) & 8.11 & $< \ell$ & 7841.67 & 1.50 & $< \ell$ & 1.00 \\
\texttt{10\_10\_3} & (3, 3, 3, 3) & 7.01 & $< \ell$ & 6473.67 & 1.67 & $< \ell$ & 1.00 \\
\texttt{10\_10\_3} & (5, 5, 5, 5) & 1.26 & $< \ell$ & 1665.33 & 1.70 & $< \ell$ & 1.00 \\
\midrule
\texttt{15\_15\_2} & (4, 4, 0, 0) & 64.46 & $< \ell$ & 7868.33 & 1.87 & $< \ell$ & 1.00 \\
\texttt{15\_15\_2} & (8, 8, 0, 0) & 4.01 & $< \ell$ & 2675.33 & 1.53 & $< \ell$ & 1.00 \\
\texttt{15\_15\_2} & (0, 0, 4, 4) & 55.83 & $< \ell$ & 6402.33 & 1.80 & $< \ell$ & 1.00 \\
\texttt{15\_15\_2} & (0, 0, 8, 8) & 5.41 & $< \ell$ & 1581.67 & 1.73 & $< \ell$ & 1.00 \\
\texttt{15\_15\_2} & (4, 4, 4, 4) & 10.63 & $< \ell$ & 4076.67 & 1.43 & $< \ell$ & 1.00 \\
\texttt{15\_15\_2} & (8, 8, 8, 8) & 0.70 & $< \ell$ & 109.67 & 1.40 & $< \ell$ & 1.00 \\
\texttt{15\_15\_3} & (4, 4, 0, 0) & 289.15 & $< \ell$ & 12407.33 & 1.90 & $< \ell$ & 1.00 \\
\texttt{15\_15\_3} & (8, 8, 0, 0) & 131.48 & $< \ell$ & 6154.00 & 1.90 & $< \ell$ & 1.00 \\
\texttt{15\_15\_3} & (0, 0, 4, 4) & 318.75 & $< \ell$ & 13601.00 & 2.27 & $< \ell$ & 1.00 \\
\texttt{15\_15\_3} & (0, 0, 8, 8) & 129.83 & $< \ell$ & 6351.67 & 1.90 & $< \ell$ & 1.00 \\
\texttt{15\_15\_3} & (4, 4, 4, 4) & 161.37 & $< \ell$ & 7175.67 & 1.90 & $< \ell$ & 1.00 \\
\texttt{15\_15\_3} & (8, 8, 8, 8) & 13.97 & $< \ell$ & 5476.33 & 2.20 & $< \ell$ & 1.00 \\
\midrule
\texttt{20\_20\_2} & (5, 5, 0, 0) & 197.11 & $< \ell$ & 4857.00 & 1.90 & $< \ell$ & 1.00 \\
\texttt{20\_20\_2} & (10, 10, 0, 0) & 57.66 & $< \ell$ & 4375.00 & 1.73 & $< \ell$ & 1.00 \\
\texttt{20\_20\_2} & (0, 0, 5, 5) & 244.00 & $< \ell$ & 4618.00 & 2.10 & $< \ell$ & 1.00 \\
\texttt{20\_20\_2} & (0, 0, 10, 10) & 30.80 & $< \ell$ & 3600.00 & 1.60 & $< \ell$ & 1.00 \\
\texttt{20\_20\_2} & (5, 5, 5, 5) & 143.75 & $< \ell$ & 7262.67 & 1.73 & $< \ell$ & 1.00 \\
\texttt{20\_20\_2} & (10, 10, 10, 10) & 3.67 & $< \ell$ & 918.33 & 1.87 & $< \ell$ & 1.00 \\
\texttt{20\_20\_3} & (5, 5, 0, 0) & 2068.58 & $< \ell$ & 22668.33 & 1.53 & $< \ell$ & 1.00 \\
\texttt{20\_20\_3} & (10, 10, 0, 0) & 456.00 & $< \ell$ & 11646.67 & 2.50 & $< \ell$ & 1.00 \\
\texttt{20\_20\_3} & (0, 0, 5, 5) & 2987.57 & 0.17 & 36275.67 & 1.80 & $< \ell$ & 1.00 \\
\texttt{20\_20\_3} & (0, 0, 10, 10) & 620.88 & $< \ell$ & 22048.00 & 2.13 & $< \ell$ & 1.00 \\
\texttt{20\_20\_3} & (5, 5, 5, 5) & 501.89 & $< \ell$ & 13801.33 & 1.53 & $< \ell$ & 1.00 \\
\texttt{20\_20\_3} & (10, 10, 10, 10) & 62.11 & $< \ell$ & 5059.67 & 2.33 & $< \ell$ & 1.00 \\
\midrule
\texttt{25\_25\_2} & (6, 6, 0, 0) & 1195.29 & $< \ell$ & 7383.67 & 1.27 & $< \ell$ & 1.00 \\
\texttt{25\_25\_2} & (13, 13, 0, 0) & 284.02 & $< \ell$ & 7962.33 & 2.37 & $< \ell$ & 1.00 \\
\texttt{25\_25\_2} & (0, 0, 6, 6) & 1629.72 & $< \ell$ & 12255.67 & 1.70 & $< \ell$ & 1.00 \\
\texttt{25\_25\_2} & (0, 0, 13, 13) & 320.09 & $< \ell$ & 9018.33 & 2.30 & $< \ell$ & 1.00 \\
\texttt{25\_25\_2} & (6, 6, 6, 6) & 635.83 & $< \ell$ & 9142.33 & 1.90 & $< \ell$ & 1.00 \\
\texttt{25\_25\_2} & (13, 13, 13, 13) & 11.90 & $< \ell$ & 3963.33 & 2.77 & $< \ell$ & 1.00 \\
\texttt{25\_25\_3} & (6, 6, 0, 0) & 3600.00 & 1.12 & 6358.00 & 1.80 & $< \ell$ & 1.00 \\
\texttt{25\_25\_3} & (13, 13, 0, 0) & 2753.72 & 0.12 & 20842.67 & 2.47 & $< \ell$ & 1.00 \\
\texttt{25\_25\_3} & (0, 0, 6, 6) & 3600.00 & 4.08 & 5545.00 & 2.07 & $< \ell$ & 1.00 \\
\texttt{25\_25\_3} & (0, 0, 13, 13) & 3600.00 & 0.36 & 33779.00 & 2.50 & $< \ell$ & 1.00 \\
\texttt{25\_25\_3} & (6, 6, 6, 6) & 3600.00 & 0.15 & 26471.67 & 1.60 & $< \ell$ & 1.00 \\
\texttt{25\_25\_3} & (13, 13, 13, 13) & 186.95 & $< \ell$ & 5934.67 & 2.17 & $< \ell$ & 1.00 \\
\bottomrule
\end{tabular}
\end{table}

Table~\ref{table:gurobi} shows the experimental results of the comparison between the proposed \texttt{CBICL-BB} algorithm and  \texttt{GUROBI} across all the instances. The reported metrics include the computational time, the optimality gap, and the branch-and-bound node count for each graph and configuration. The results are averaged over three random seeds for each constraint configuration. In column ``Avg. Gap (\%)'', the symbol ``$< \ell$'' indicates that all instances for the given constraint configuration were solved within the required optimality tolerance. We set a time limit of 3600 seconds and report the average gap achieved along with the number of explored nodes upon reaching this limit. 

Overall, \texttt{CBICL-BB} consistently solves all instances within the required optimality tolerance, and it does so by exploring a single node in all cases. Moreover, these instances are solved without performing any cutting-plane iterations, as the basic SDP relaxation is already tight. This phenomenon aligns with the theoretical findings in \cite{ames2014guaranteed}, which show that the SDP relaxation of the unconstrained $k$-DDB problem is tight when the edge weights of the input graph are highly concentrated within a few disjoint subgraphs. In contrast, while \texttt{GUROBI} also achieves optimality in the majority of instances, it requires significantly more branching and computational time, particularly as problem size and number of planted bicliques increase. For example, on the instance \texttt{25\_25\_3} with column-side constraints only (0, 0, 6, 6), \texttt{GUROBI} reaches the time limit of 3600 seconds and reports an average optimality gap of 4.08\%. Similar slowdowns or timeouts are observed on the largest instances, although the achieved optimality gaps upon reaching the time limit are very small. A relevant subset of these hard cases are instances with three planted bicliques and unbalanced constraint application—e.g., constraints applied only on one side of the bipartite graph. These cases are more challenging for \texttt{GUROBI}, as the search space remains large. Conversely, \texttt{CBICL-BB}'s performance remains unaffected by size and the number of constraints, solving all such instances in under 3 seconds and directly at the root node. Importantly, we observe that increasing the number of must-link constraints has a beneficial effect on the computational time. Since must-link constraints effectively reduce the number of decision variables by merging nodes into the same cluster, the feasible space becomes smaller and the number of variables decreases. This reduces the burden on both solvers, but it is particularly noticeable in \texttt{GUROBI}'s behavior: in many cases, higher numbers of must-link constraints lead to lower node counts and faster solution times. For instance, for \texttt{20\_20\_2}, moving from $(5, 5, 0, 0)$ to $(10, 10, 0, 0)$ cuts solution time from 197.11s to 57.66s and reduces node count by over 10\%. Moreover, across all the instances, moving from configurations with $(n/4, n/4, m/4, m/4)$ to $(n/2, n/2, m/2, m/2)$ constraints results in a sharp decrease in runtime. As the instance size increases, the performance gap between solvers widens. For instance, \texttt{GUROBI} often exceeds 1000 seconds on \texttt{25\_25\_3}, while \texttt{CBICL-BB} solves all instances in under 3 seconds. Notably, even for configurations where \texttt{GUROBI} times out or returns a gap above the required tolerance, \texttt{CBICL-BB} still delivers an optimal solution efficiently.

The results demonstrate the substantial computational advantage of \texttt{CBICL-BB} over \texttt{GUROBI}, especially on the larger instances of the test bed. 
%Its ability to consistently solve all problems to global optimality at the root node underscores the effectiveness of a tailored approach for the structured QCQP arising in constrained biclustering. 
Similar trends were observed when using \texttt{BARON} (version 23.3.11) \cite{sahinidis1996baron}, another well-established solver for non-convex global optimization. Although \texttt{BARON} achieved comparable behavior to \texttt{GUROBI}, its computing timings were generally larger. For this reason, detailed results for \texttt{BARON} are omitted from Table~\ref{table:gurobi}. We finally stress that both \texttt{GUROBI} and \texttt{BARON} are general-purpose solvers and do not fully take advantage of the continuous and discrete nature of Problem~\eqref{prob:or_constr_shr}. Their relatively poor performance on this class of instances is therefore expected and highlights the value of exploiting structure through specialized algorithms like \texttt{CBICL-BB}. In the next section, we will extensively test \texttt{CBICL-BB} on real-world instances whose sizes far exceed those considered in artificial graphs.

%We emphasize that both \texttt{GUROBI} and \texttt{BARON} are general-purpose solvers that do not incorporate problem-specific techniques for Problem~\eqref{prob:or}. 
%Their comparatively lower performance is therefore expected and further illustrates the benefits of tailored algorithms such as \texttt{CBICL-BB}.
%Additionally, we observe that increasing the number of must-link constraints simplifies the optimization problem by effectively reducing its dimensionality. Since must-link constraints merge vertices into joint clusters, they reduce the number of variables and shrink the feasible region. This is reflected in the data: for example, in \texttt{20\_20\_2}, moving from $(5, 5, 0, 0)$ to $(10, 10, 0, 0)$ results in a sharp decrease in runtime (197.11s to 57.66s) and fewer nodes explored (from 4857 to 4375). The trend is consistent across instance sizes, benefiting both solvers, though the improvement is more impactful for \texttt{GUROBI}.
%Finally, constraint density also influences solver performance. Denser and symmetric constraints (e.g., $(n/2, n/2, n/2, n/2)$) help reduce ambiguity and symmetry in the solution space, improving \texttt{GUROBI}'s runtime and pruning ability. However, \texttt{CBICL-BB}'s performance remains stable across all configurations, demonstrating its robustness to constraint sparsity and imbalance.

\subsection{Experiments on gene expression instances}
To better evaluate the effectiveness of $\texttt{CBICL-BB}$, we consider real-world gene expression datasets from \cite{de2008clustering, sudoso2024}. This publicly available benchmark collection comprises microarray data matrices, with samples on the rows and conditions of the columns, derived from experiments on cancer gene expression. All datasets have been preprocessed by removing noninformative genes, i.e., genes that do not display differential expression across samples. In the resulting bipartite graphs, one set of nodes represents the samples and the other set represents the genes, with edges indicating the associated expression levels. A summary of gene expression datasets used in the experiment is reported in Table \ref{table:gene_data}.

\begin{table}[!ht]
\footnotesize
 \caption{Summary of gene expression datasets used in the experiments. For each dataset, we report the number of samples ($n$), the number of conditions ($m$), the total number of vertices ($n + m$) in the corresponding bipartite graph, the target number of biclusters ($k$), and the constraint configurations.}
    \label{table:gene_data}
    \centering
    \begin{tabular}{llccccl}
    \toprule
        ID & Dataset &  $n$ & $m$ & $n+m$ & $k$ & Constr. $(|\textrm{ML}_U|, |\textrm{CL}_U|, |\textrm{ML}_V|, |\textrm{CL}_V|)$\\
        \midrule
\multirow{2}{*}{1} & \multirow{2}{*}{\texttt{Bhattacharjee-2001}} & \multirow{2}{*}{193} & \multirow{2}{*}{203} & \multirow{2}{*}{396} & \multirow{2}{*}{3} & (48, 48, 0, 0) (0, 0, 51, 51) (48, 48, 51, 51)\\
%& &  &  &  &            & (97, 97, 0, 0) (0, 0, 102, 102) (97, 97, 102, 102) \\ \midrule
%\multirow{2}{*}{2} & \multirow{2}{*}{Dyrskjot-2003}	&	\multirow{2}{*}{301}	&	\multirow{2}{*}{40}	& \multirow{2}{*}{341} &	\multirow{2}{*}{2} & (75, 75, 0, 0) (0, 0, 10, 10) (75, 75, 10, 10)	\\
&  &  &  &  &            & (97, 97, 0, 0) (0, 0, 102, 102) (97, 97, 102, 102) \\ \midrule
\multirow{2}{*}{2} & \multirow{2}{*}{\texttt{Golub-1999}}	&	\multirow{2}{*}{467}	&	\multirow{2}{*}{72}	& \multirow{2}{*}{539}	& \multirow{2}{*}{2} &	(117, 117, 0, 0) (0, 0, 18, 18) (117, 117, 18, 18)\\
& &  &  &  &            & (234, 234, 0, 0) (0, 0, 36, 36) (234, 234, 36, 36) \\ \midrule
\multirow{2}{*}{3} & \multirow{2}{*}{\texttt{Khan-2001}}	&	\multirow{2}{*}{267}	&	\multirow{2}{*}{83}	& \multirow{2}{*}{350}	& \multirow{2}{*}{2} & (67, 67, 0, 0) (0, 0, 21, 21) (67, 67, 21, 21)	\\
& &  &  &  &            & (134, 134, 0, 0) (0, 0, 42, 42) (134, 134, 42, 42) \\  \midrule
\multirow{2}{*}{4} & \multirow{2}{*}{\texttt{Pomeroy-2002}}	&	\multirow{2}{*}{214}	&	\multirow{2}{*}{34}	& \multirow{2}{*}{248}	& \multirow{2}{*}{7}	 & (54, 54, 0, 0) (0, 0, 9, 9) (54, 54, 9, 9)\\
& &  &  &  &            & (107, 107, 0, 0) (0, 0, 17, 17) (107, 107, 17, 17) \\ \midrule
\multirow{2}{*}{5} & \multirow{2}{*}{\texttt{Ramaswamy-2001}}	&	\multirow{2}{*}{341}	&	\multirow{2}{*}{190}	& \multirow{2}{*}{531} &	\multirow{2}{*}{2} & (85, 85, 0, 0) (0, 0, 48, 48) (85, 85, 48, 48)\\
& &  &  &  &            & (171, 171, 0, 0) (0, 0, 95, 95) (171, 171, 95, 95)\\ \midrule
\multirow{2}{*}{6} & \multirow{2}{*}{\texttt{Singh-2002}}	&	\multirow{2}{*}{169}	&	\multirow{2}{*}{102}	& \multirow{2}{*}{271} & \multirow{2}{*}{3}	& (42, 42, 0, 0) (0, 0, 26, 26) (42, 42, 26, 26)\\
& &  &  &  &            & (85, 85, 0, 0) (0, 0, 51, 51) (85, 85, 51, 51) \\
\bottomrule
\end{tabular}
\end{table}

Constraint configurations are reported as 4-tuples $(|\textrm{ML}_U|, |\textrm{CL}_U|, |\textrm{ML}_V|, |\textrm{CL}_V|)$, indicating the number of must-link and cannot-link constraints applied to the sample side ($U$) and the condition side ($V$), respectively. For each configuration, we generate three distinct sets of pairwise constraints using different random seeds, resulting in multiple instances per dataset that vary in both the type and amount of background knowledge. 
As is customary in the literature \cite{wagstaff2000clustering}, we generate constraints from a ground-truth reference labeling by randomly selecting pairs of vertices and assigning a must-link constraint if they belong to the same cluster, and a cannot-link constraint otherwise. The ground-truth labeling for each dataset is obtained from the optimal solution of the corresponding unconstrained biclustering problem. However, as observed in several studies on semi-supervised learning \cite{randel2021lagrangian, piccialli2022exact}, the inclusion of background knowledge that conflicts with this reference solution can significantly increase the computational complexity of the problem. In these cases, the solver must explore a more difficult search space to enforce the constraints while identifying a globally optimal solution. To evaluate the robustness of the proposed exact solver under this scenario, we introduce constraint sets that are, by construction, in partial disagreement with the unconstrained solution. Specifically, we generate several configurations where a fixed proportion of the constraints—10\%, 20\%, 30\%, and 40\%—are violated with respect to the unconstrained optimal biclustering. Thus, we consider scenarios in which external supervision (e.g., previous knowledge or annotations from domain experts) may rightfully disagree with the unsupervised (unconstrained) solution, potentially guiding the solver toward a more meaningful or interpretable biclustering structure. This setup reflects real-world use cases where the unconstrained optimum may not capture all relevant structure, and the constraints are intended to steer the solution process appropriately, even at the cost of increased computational effort. An interesting discussion about this is given in \cite{randel2021lagrangian}. Additionally, we study the impact of how constraints are distributed across the two sides of the bipartite graph. We consider three settings: (i) constraints applied only on the sample set $U$, with sizes $(n/4, n/4, 0, 0)$ and $(n/2, n/2, 0, 0)$; (ii) constraints applied only on the condition set $V$, with sizes $(0, 0, m/4, m/4)$ and $(0, 0, m/2, m/2)$; and (iii) constraints applied on both sides, with sizes $(n/4, n/4, m/4, m/4)$ and $(n/2, n/2, m/2, m/2)$. All values are rounded to the nearest integer using half-up rounding. Overall, the test set comprises 6 real-world datasets, each evaluated under 4 levels of constraint violation (10\%, 20\%, 30\%, 40\%), 6 distinct constraint configurations, and 3 random seeds per configuration. This results in a total of 432 constrained biclustering instances.

Table~\ref{table:gene_results} reports the computational results of the exact solver {\tt CBICL-BB} and the heuristic {\tt CBICL-LR}. The table reports, for each dataset, constraint configuration, and violation level, the average performance over three random seeds under different metrics. We set a time limit of 10800 seconds and report the average gap achieved along with the number of explored nodes upon reaching this limit. For the exact solver \texttt{CBICL-BB}, the column ``Gap$_0$'' denotes the optimality gap (in percentage) at the root node, while ``CP$_0$'' indicates the number of cutting-plane iterations performed at the root before branching. The column ``Gap'' refers to the final optimality gap at termination—either when the global optimum is proven or all three seeds under that configuration (reported as ``$< \ell$'') or when the time limit is reached. ``Nodes'' indicates the number of nodes explored in the branch-and-bound tree. The column ``ILP'' reports the cumulative time (in seconds) spent solving all integer linear programming subproblems inside the rounding heuristic. Finally, ``Time'' gives the total runtime of \texttt{CBICL-BB}.

{\footnotesize
\begin{longtable}{
    >{\raggedright\arraybackslash}p{0.05cm}
    >{\raggedright\arraybackslash}p{2.08cm}
    >{\raggedright\arraybackslash}p{0.45cm}
    >{\raggedleft\arraybackslash}p{0.6cm} % Gap_0
    >{\raggedleft\arraybackslash}p{0.65cm} % CP_0
    >{\raggedleft\arraybackslash}p{0.45cm} % Gap
    >{\raggedleft\arraybackslash}p{0.75cm} % Nodes
    >{\raggedleft\arraybackslash}p{0.65cm} % ILP
    >{\raggedleft\arraybackslash}p{0.95cm} % Time
    >{\raggedleft\arraybackslash}p{0.6cm} % Time
    >{\raggedleft\arraybackslash}p{0.5cm} % Gap^*
    r
}
\caption{Results of {\tt CBICL-BB} and {\tt CBICL-LR} on real-world gene expression datasets under varying constraint configurations (Constr.) and violation percentages (Viol.). Results are averaged over three sets of constraints for each configuration. For {\tt CBICL-BB}, we report the root node gap (Gap$_0$), the number of cutting-plane iterations at the root (CP$_0$), final optimality gap (Gap), number of nodes explored (Nodes), cumulative ILP solve time (ILP), and total solution time (Time). For {\tt CBICL-LR}, we report the total runtime (Time) and the gap (Gap$^\star$) compared to the objective function value found by {\tt CBICL-BB}.}
\label{table:gene_results} \\
\toprule
\multicolumn{3}{c}{} &
\multicolumn{6}{c}{\texttt{CBICL-BB}} &
\multicolumn{2}{c}{\texttt{CBICL-LR}} \\
\cmidrule(lr){4-9} \cmidrule(lr){10-11}
\multicolumn{3}{c}{} &
\multicolumn{6}{c}{Avg.} &
\multicolumn{2}{c}{Avg.} \\
%\cmidrule(lr){4-9} \cmidrule(lr){10-12}
ID & Constr. & Viol. {\tiny (\%)} & $\textrm{Gap}_0$ {\tiny (\%)} & $\textrm{CP}_0$ & Gap {\tiny (\%)} & Nodes & ILP {\tiny (s)} & Time {\tiny (s)} & Time {\tiny (s)} & $\textrm{Gap}^\star$ {\tiny (\%)}\\
\midrule
\endfirsthead

\toprule
\multicolumn{3}{c}{} &
\multicolumn{6}{c}{\texttt{CBICL-BB}} &
\multicolumn{2}{c}{\texttt{CBICL-LR}} \\
\cmidrule(lr){4-9} \cmidrule(lr){10-11}
\multicolumn{3}{c}{} &
\multicolumn{6}{c}{Avg.} &
\multicolumn{2}{c}{Avg.} \\
%\cmidrule(lr){4-9} \cmidrule(lr){10-12}
ID & Constr. & Viol. {\tiny (\%)} & $\textrm{Gap}_0$ {\tiny (\%)} & $\textrm{CP}_0$ & Gap {\tiny (\%)} & Nodes & ILP {\tiny (s)} & Time {\tiny (s)} & Time {\tiny (s)} & $\textrm{Gap}^\star$ {\tiny (\%)}\\
\midrule
\endhead

\midrule
\multicolumn{11}{r}{\textit{Continued on next page}} \\
\midrule
\endfoot

\bottomrule
\endlastfoot

1	&	(48, 48, 0, 0)	&	10	&	0.11	&	7.33	&	$< \ell$	&	1.67	&	0.13	&	543.45	&	32.26	&	0.28	\\
1	&	(97, 97, 0, 0)	&	10	&	$<\ell$	&	6.67	&	$<\ell$	&	1.67	&	0.11	&	299.67	&	31.38	&	1.06	\\
1	&	(0, 0, 51, 51)	&	10	&	0.14	&	9.67	&	$<\ell$	&	14.00	&	0.64	&	997.87	&	30.40	&	0.23	\\
1	&	(0, 0, 102, 102)	&	10	&	0.31	&	10.33	&	$<\ell$	&	23.67	&	1.21	&	1537.67	&	36.71	&	0.59	\\
1	&	(48, 48, 51, 51)	&	10	&	$<\ell$	&	6.33	&	$<\ell$	&	1.00	&	0.11	&	448.67	&	25.85	&	0.53	\\
1	&	(97, 97, 102, 102)	&	10	&	0.24	&	8.00	&	$<\ell$	&	17.33	&	1.01	&	1349.67	&	16.50	&	1.66	\\
1	&	(48, 48, 0, 0)	&	20	&	0.13	&	6.67	&	$<\ell$	&	2.33	&	0.14	&	529.31	&	28.89	&	1.49	\\
1	&	(97, 97, 0, 0)	&	20	&	$<\ell$	&	4.67	&	$<\ell$	&	1.67	&	0.10	&	183.69	&	32.99	&	0.54	\\
1	&	(0, 0, 51, 51)	&	20	&	0.22	&	10.67	&	$<\ell$	&	24.33	&	1.10	&	1172.61	&	27.77	&	0.47	\\
1	&	(0, 0, 102, 102)	&	20	&	0.31	&	10.33	&	$<\ell$	&	15.67	&	0.84	&	1385.34	&	42.66	&	3.51	\\
1	&	(48, 48, 51, 51)	&	20	&	0.11	&	7.67	&	$<\ell$	&	2.33	&	0.16	&	1545.36	&	24.81	&	1.35	\\
1	&	(97, 97, 102, 102)	&	20	&	0.27	&	7.67	&	$<\ell$	&	9.33	&	0.73	&	871.32	&	16.22	&	1.18	\\
1	&	(48, 48, 0, 0)	&	30	&	0.14	&	7.33	&	$<\ell$	&	1.67	&	0.12	&	450.24	&	33.07	&	0.72	\\
1	&	(97, 97, 0, 0)	&	30	&	0.17	&	8.00	&	$<\ell$	&	2.33	&	0.15	&	657.11	&	37.42	&	2.13	\\
1	&	(0, 0, 51, 51)	&	30	&	0.17	&	10.33	&	$<\ell$	&	7.00	&	0.32	&	799.67	&	29.32	&	0.64	\\
1	&	(0, 0, 102, 102)	&	30	&	1.59	&	11.33	&	$<\ell$	&	32.00	&	1.87	&	2576.67	&	56.77	&	0.99	\\
1	&	(48, 48, 51, 51)	&	30	&	0.12	&	7.67	&	$<\ell$	&	1.67	&	0.14	&	453.38	&	24.68	&	1.54	\\
1	&	(97, 97, 102, 102)	&	30	&	0.61	&	10.33	&	$<\ell$	&	19.00	&	1.41	&	2393.64	&	16.20	&	2.79	\\
1	&	(48, 48, 0, 0)	&	40	&	0.79	&	8.67	&	$<\ell$	&	29.33	&	1.02	&	1627.74	&	33.62	&	2.24	\\
1	&	(97, 97, 0, 0)	&	40	&	3.52	&	11.33	&	0.49	&	58.00	&	4.02	&	10800	&	32.57	&	5.39	\\
1	&	(0, 0, 51, 51)	&	40	&	0.39	&	11.00	&	$<\ell$	&	11.00	&	0.64	&	1300.75	&	30.76	&	0.96	\\
1	&	(0, 0, 102, 102)	&	40	&	2.20	&	12.33	&	$<\ell$	&	54.33	&	4.44	&	5307.33	&	49.21	&	1.49	\\
1	&	(48, 48, 51, 51)	&	40	&	0.72	&	8.67	&	$<\ell$	&	34.67	&	0.98	&	1600.33	&	22.52	&	3.16	\\
1	&	(97, 97, 102, 102)	&	40	&	6.02	&	12.67	&	$<\ell$	&	53.33	&	5.56	&	7417.67	&	17.99	&	5.87	\\
\midrule																					
2	&	(117, 117, 0, 0)	&	10	&	$<\ell$	&	17.00	&	$<\ell$	&	1.00	&	0.14	&	3025.42	&	43.95	&	2.84	\\
2	&	(234, 234, 0, 0)	&	10	&	0.19	&	10.67	&	$<\ell$	&	2.33	&	0.12	&	2223.18	&	28.50	&	3.62	\\
2	&	(0, 0, 18, 18)	&	10	&	0.20	&	20.00	&	$<\ell$	&	17.33	&	1.25	&	4204.43	&	56.38	&	0.32	\\
2	&	(0, 0, 36, 36)	&	10	&	1.37	&	15.67	&	$<\ell$	&	14.00	&	1.57	&	4098.10	&	60.28	&	1.61	\\
2	&	(117, 117, 18, 18)	&	10	&	0.16	&	20.67	& 	$<\ell$	&	4.67	&	0.33	&	5337.33	&	44.49	&	2.72	\\
2	&	(234, 234, 36, 36)	&	10	&	0.49	&	11.33	&	$<\ell$	&	23.00	&	1.21	&	5328.34	&	25.15	&	4.42	\\
2	&	(117, 117, 0, 0)	&	20	&	$<\ell$	&	19.00	&	$<\ell$	&	1.67	&	0.16	&	2709.35	&	46.04	&	2.95	\\
2	&	(234, 234, 0, 0)	&	20	&	0.35	&	8.67	&	$<\ell$	&	4.33	&	0.22	&	2455.84	&	27.44	&	5.15	\\
2	&	(0, 0, 18, 18)	&	20	&	0.97	&	16.00	&	$<\ell$	&	23.67	&	1.59	&	8256.41	&	67.01	&	0.15	\\
2	&	(0, 0, 36, 36)	&	20	&	0.86	&	24.00	&	$<\ell$	&	33.67	&	2.82	&	9022.67	&	77.62	&	0.58	\\
2	&	(117, 117, 18, 18)	&	20	&	0.29	&	19.67	&	$<\ell$	&	7.67	&	0.35	&	3627.02	&	46.33	&	2.57	\\
2	&	(234, 234, 36, 36)	&	20	&	0.23	&	11.00	&	$<\ell$	&	24.33	&	1.08	&	4656.33	&	37.75	&	4.59	\\
2	&	(117, 117, 0, 0)	&	30	&	0.11	&	24.33	&	$<\ell$	&	3.67	&	0.26	&	3741.33	&	35.03	&	5.71	\\
2	&	(234, 234, 0, 0)	&	30	&	0.20	&	8.33	&	$<\ell$	&	8.67	&	0.47	&	4997.67	&	26.25	&	5.04	\\
2	&	(0, 0, 18, 18)	&	30	&	0.69	&	24.00	&	$<\ell$	&	22.67	&	2.08	&	7419.67	&	56.53	&	0.17	\\
2	&	(0, 0, 36, 36)	&	30	&	0.86	&	23.67	&	$<\ell$	&	20.67	&	3.49	&	9730.32	&	88.38	&	0.48	\\
2	&	(117, 117, 18, 18)	&	30	&	0.13	&	23.33	&	$<\ell$	&	3.00	&	0.24	&	3426.33	&	40.23	&	5.49	\\
2	&	(234, 234, 36, 36)	&	30	&	0.18	&	9.33	&	$<\ell$	&	23.00	&	0.87	&	4385.11	&	39.74	&	4.87	\\
2	&	(117, 117, 0, 0)	&	40	&	0.11	&	23.00	&	$<\ell$	&	6.33	&	0.46	&	3501.67	&	33.87	&	5.29	\\
2	&	(234, 234, 0, 0)	&	40	&	0.14	&	10.00	&	$<\ell$	&	2.33	&	0.12	&	3229.33	&	25.74	&	4.76	\\
2	&	(0, 0, 18, 18)	&	40	&	0.73	&	21.71	&	$<\ell$	&	46.00	&	2.26	&	7355.67	&	56.58	&	0.99	\\
2	&	(0, 0, 36, 36)	&	40	&	0.70	&	22.33	&	0.47	&	25.67	&	3.07	&	10800	&	73.69	&	0.14	\\
2	&	(117, 117, 18, 18)	&	40	&	0.12	&	22.39	&	$<\ell$	&	5.67	&	0.39	&	3825.50	&	41.65	&	4.41	\\
2	&	(234, 234, 36, 36)	&	40	&	0.20	&	8.67	&	$<\ell$	&	29.33	&	1.31	&	5101.33	&	32.16	&	4.74	\\
\midrule																					
3	&	(67, 67, 0, 0)	&	10	&	$<\ell$	&	6.33	&	$<\ell$	&	1.00	&	0.05	&	193.33	&	2.69	&	0.41	\\
3	&	(134, 134, 0, 0)	&	10	&	0.11	&	6.33	&	$<\ell$	&	1.67	&	0.06	&	364.33	&	2.87	&	1.28	\\
3	&	(0, 0, 21, 21)	&	10	&	0.14	&	8.00	&	$<\ell$	&	12.33	&	0.32	&	934.67	&	3.84	&	0.07	\\
3	&	(0, 0, 42, 42)	&	10	&	0.11	&	10.33	&	$<\ell$	&	5.67	&	0.19	&	848.33	&	3.63	&	0.53	\\
3	&	(67, 67, 21, 21)	&	10	&	$<\ell$	&	6.33	&	$<\ell$	&	1.00	&	0.05	&	248.01	&	2.83	&	0.34	\\
3	&	(134, 134, 42, 42)	&	10	&	0.28	&	7.00	&	$<\ell$	&	18.67	&	0.72	&	1716.33	&	2.53	&	1.53	\\
3	&	(67, 67, 0, 0)	&	20	&	$<\ell$	&	8.00	&	$<\ell$	&	1.00	&	0.06	&	349.40	&	2.61	&	0.78	\\
3	&	(134, 134, 0, 0)	&	20	&	0.17	&	8.00	&	$<\ell$	&	3.67	&	0.12	&	762.67	&	2.52	&	0.81	\\
3	&	(0, 0, 21, 21)	&	20	&	0.18	&	10.33	&	$<\ell$	&	4.33	&	0.16	&	409.03	&	3.88	&	0.18	\\
3	&	(0, 0, 42, 42)	&	20	&	0.24	&	9.33	&	$<\ell$	&	12.00	&	0.34	&	782.67	&	4.31	&	1.11	\\
3	&	(67, 67, 21, 21)	&	20	&	0.12	&	9.00	&	$<\ell$	&	3.00	&	0.11	&	531.61	&	2.68	&	0.96	\\
3	&	(134, 134, 42, 42)	&	20	&	0.21	&	7.33	&	$<\ell$	&	34.67	&	0.89	&	1937.33	&	2.51	&	1.44	\\
3	&	(67, 67, 0, 0)	&	30	&	0.11	&	9.00	&	$<\ell$	&	1.67	&	0.08	&	433.65	&	2.96	&	1.18	\\
3	&	(134, 134, 0, 0)	&	30	&	0.45	&	8.33	&	$<\ell$	&	27.00	&	0.69	&	4040.21	&	2.13	&	1.75	\\
3	&	(0, 0, 21, 21)	&	30	&	0.22	&	12.33	&	$<\ell$	&	4.33	&	0.16	&	468.67	&	3.63	&	1.36	\\
3	&	(0, 0, 42, 42)	&	30	&	0.22	&	11.33	&	$<\ell$	&	15.33	&	0.53	&	1601.11	&	5.56	&	1.16	\\
3	&	(67, 67, 21, 21)	&	30	&	0.32	&	9.67	&	$<\ell$	&	11.33	&	0.35	&	942.33	&	2.38	&	2.06	\\
3	&	(134, 134, 42, 42)	&	30	&	0.15	&	8.00	&	$<\ell$	&	17.00	&	0.55	&	1499.08	&	2.98	&	1.21	\\
3	&	(67, 67, 0, 0)	&	40	&	1.16	&	11.00	&	$<\ell$	&	31.67	&	0.89	&	3068.33	&	2.65	&	5.12	\\
3	&	(134, 134, 0, 0)	&	40	&	0.74	&	9.33	&	$<\ell$	&	29.00	&	0.85	&	6188.67	&	2.19	&	1.52	\\
3	&	(0, 0, 21, 21)	&	40	&	0.58	&	14.67	&	$<\ell$	&	45.33	&	1.07	&	1926.07	&	4.03	&	4.46	\\
3	&	(0, 0, 42, 42)	&	40	&	0.14	&	8.33	&	$<\ell$	&	5.33	&	1.85	&	1499.36	&	7.65	&	1.22	\\
3	&	(67, 67, 21, 21)	&	40	&	0.13	&	9.00	&	$<\ell$	&	5.67	&	0.21	&	596.44	&	2.93	&	3.41	\\
3	&	(134, 134, 42, 42)	&	40	&	0.17	&	5.33	&	$<\ell$	&	9.00	&	0.28	&	1139.67	&	7.73	&	1.15	\\
\midrule																					
4	&	(54, 54, 0, 0)	&	10	&	2.04	&	11.33	&	$<\ell$	&	80.33	&	9.07	&	3258.30	&	34.88	&	0.20	\\
4	&	(107, 107, 0, 0)	&	10	&	2.88	&	8.67	&	$<\ell$	&	123.00	&	5.65	&	4342.67	&	13.42	&	0.67	\\
4	&	(0, 0, 9, 9)	&	10	&	0.37	&	6.67	&	$<\ell$	&	6.33	&	0.29	&	209.35	&	36.39	&	0.92	\\
4	&	(0, 0, 17, 17)	&	10	&	1.14	&	7.33	&	$<\ell$	&	23.00	&	1.12	&	553.81	&	21.24	&	0.17	\\
4	&	(54, 54, 9, 9)	&	10	&	1.86	&	11.67	&	$<\ell$	&	56.33	&	6.91	&	2314.33	&	21.96	&	0.20	\\
4	&	(107, 107, 17, 17)	&	10	&	4.05	&	8.67	&	$<\ell$	&	111.67	&	5.46	&	4388.67	&	13.43	&	1.11	\\
4	&	(54, 54, 0, 0)	&	20	&	5.20	&	10.00	&	$<\ell$	&	91.00	&	10.65	&	3439.67	&	36.09	&	0.71	\\
4	&	(107, 107, 0, 0)	&	20	&	6.63	&	9.00	&	$<\ell$	&	87.67	&	5.71	&	4986.33	&	13.57	&	1.81	\\
4	&	(0, 0, 9, 9)	&	20	&	0.95	&	7.00	&	$<\ell$	&	14.33	&	0.72	&	357.01	&	38.66	&	2.25	\\
4	&	(0, 0, 17, 17)	&	20	&	1.18	&	7.00	&	$<\ell$	&	18.33	&	0.93	&	399.33	&	44.18	&	0.14	\\
4	&	(54, 54, 9, 9)	&	20	&	6.49	&	10.33	&	$<\ell$	&	104.00	&	11.07	&	3791.20	&	18.52	&	0.77	\\
4	&	(107, 107, 17, 17)	&	20	&	5.45	&	9.00	&	$<\ell$	&	119.67	&	5.33	&	4866.68	&	12.81	&	1.74	\\
4	&	(54, 54, 0, 0)	&	30	&	8.29	&	10.00	&	$<\ell$	&	75.33	&	10.61	&	3021.33	&	51.78	&	1.55	\\
4	&	(107, 107, 0, 0)	&	30	&	9.26	&	8.33	&	$<\ell$	&	106.33	&	5.91	&	4698.02	&	14.23	&	1.57	\\
4	&	(0, 0, 9, 9)	&	30	&	1.43	&	7.67	&	$<\ell$	&	41.67	&	2.52	&	908.31	&	29.52	&	0.05	\\
4	&	(0, 0, 17, 17)	&	30	&	2.64	&	8.00	&	$<\ell$	&	45.00	&	2.95	&	891.33	&	56.27	&	0.90	\\
4	&	(54, 54, 9, 9)	&	30	&	9.91	&	9.00	&	$<\ell$	&	78.67	&	10.62	&	3186.67	&	24.01	&	2.81	\\
4	&	(107, 107, 17, 17)	&	30	&	9.87	&	7.33	&	$<\ell$	&	116.33	&	6.23	&	5120.07	&	13.80	&	2.47	\\
4	&	(54, 54, 0, 0)	&	40	&	7.27	&	8.33	&	$<\ell$	&	103.67	&	10.93	&	3314.81	&	35.04	&	2.19	\\
4	&	(107, 107, 0, 0)	&	40	&	7.03	&	8.33	&	$<\ell$	&	98.67	&	5.99	&	4185.39	&	14.68	&	4.90	\\
4	&	(0, 0, 9, 9)	&	40	&	2.11	&	7.67	&	$<\ell$	&	48.33	&	2.87	&	995.63	&	67.14	&	0.01	\\
4	&	(0, 0, 17, 17)	&	40	&	3.77	&	7.00	&	$<\ell$	&	19.00	&	0.96	&	567.34	&	58.75	&	0.17	\\
4	&	(54, 54, 9, 9)	&	40	&	9.46	&	9.00	&	$<\ell$	&	87.00	&	8.32	&	3081.62	&	51.22	&	2.19	\\
4	&	(107, 107, 17, 17)	&	40	&	9.37	&	8.33	&	$<\ell$	&	113.00	&	3.78	&	4681.67	&	27.96	&	4.73	\\
\midrule																					
5	&	(85, 85, 0, 0)	&	10	&	$<\ell$	&	4.33	&	$<\ell$	&	1.00	&	0.05	&	162.03	&	34.47	&	0.55	\\
5	&	(171, 171, 0, 0)	&	10	&	$<\ell$	&	4.33	&	$<\ell$	&	1.00	&	0.05	&	237.34	&	33.97	&	1.31	\\
5	&	(0, 0, 48, 48)	&	10	&	0.14	&	16.33	&	$<\ell$	&	7.00	&	0.31	&	4391.35	&	41.16	&	0.17	\\
5	&	(0, 0, 95, 95)	&	10	&	0.11	&	15.00	&	$<\ell$	&	1.67	&	0.13	&	2576.24	&	56.51	&	1.26	\\
5	&	(85, 85, 48, 48)	&	10	&	0.11	&	7.33	&	$<\ell$	&	1.67	&	0.08	&	1099.30	&	32.67	&	0.57	\\
5	&	(171, 171, 95, 95)	&	10	&	0.15	&	10.00	&	$<\ell$	&	14.33	&	0.95	&	3157.67	&	28.61	&	2.27	\\
5	&	(85, 85, 0, 0)	&	20	&	$<\ell$	&	5.00	&	$<\ell$	&	1.00	&	0.05	&	237.61	&	34.24	&	0.68	\\
5	&	(171, 171, 0, 0)	&	20	&	$<\ell$	&	4.67	&	$<\ell$	&	1.00	&	0.05	&	409.67	&	29.85	&	1.43	\\
5	&	(0, 0, 48, 48)	&	20	&	$<\ell$	&	16.67	&	$<\ell$	&	1.00	&	0.13	&	3949.30	&	55.43	&	1.55	\\
5	&	(0, 0, 95, 95)	&	20	&	0.13	&	19.00	&	$<\ell$	&	3.00	&	0.25	&	5328.84	&	46.94	&	2.67	\\
5	&	(85, 85, 48, 48)	&	20	&	$<\ell$	&	8.00	&	$<\ell$	&	1.00	&	0.07	&	753.33	&	43.82	&	1.83	\\
5	&	(171, 171, 95, 95)	&	20	&	0.20	&	7.67	&	$<\ell$	&	34.67	&	1.01	&	4411.69	&	24.69	&	3.38	\\
5	&	(85, 85, 0, 0)	&	30	&	$<\ell$	&	5.67	&	$<\ell$	&	1.00	&	0.06	&	371.08	&	46.68	&	0.99	\\
5	&	(171, 171, 0, 0)	&	30	&	$<\ell$	&	4.33	&	$<\ell$	&	1.00	&	0.05	&	304.32	&	26.88	&	0.53	\\
5	&	(0, 0, 48, 48)	&	30	&	0.13	&	18.33	&	$<\ell$	&	3.67	&	0.22	&	4687.33	&	56.96	&	1.84	\\
5	&	(0, 0, 95, 95)	&	30	&	0.15	&	18.67	&	$<\ell$	&	3.00	&	0.21	&	6084.64	&	50.18	&	3.97	\\
5	&	(85, 85, 48, 48)	&	30	&	0.16	&	11.00	&	$<\ell$	&	16.33	&	0.46	&	2968.36	&	43.05	&	2.27	\\
5	&	(171, 171, 95, 95)	&	30	&	0.19	&	8.00	&	$<\ell$	&	26.00	&	0.82	&	5324.71	&	27.56	&	3.36	\\
5	&	(85, 85, 0, 0)	&	40	&	$<\ell$	&	6.00	&	$<\ell$	&	1.00	&	0.06	&	343.37	&	46.21	&	1.67	\\
5	&	(171, 171, 0, 0)	&	40	&	$<\ell$	&	4.67	&	$<\ell$	&	4.33	&	0.12	&	728.62	&	26.94	&	0.55	\\
5	&	(0, 0, 48, 48)	&	40	&	0.44	&	23.33	&	$<\ell$	&	6.33	&	1.21	&	7354.38	&	54.56	&	3.88	\\
5	&	(0, 0, 95, 95)	&	40	&	0.28	&	24.67	&	$<\ell$	&	5.67	&	0.38	&	9575.39	&	50.82	&	4.05	\\
5	&	(85, 85, 48, 48)	&	40	&	0.13	&	9.67	&	$<\ell$	&	8.33	&	0.37	&	4124.67	&	43.66	&	4.48	\\
5	&	(171, 171, 95, 95)	&	40	&	0.17	&	8.33	&	$<\ell$	&	10.00	&	0.41	&	3015.30	&	26.79	&	3.71	\\
\midrule																
6	&	(42, 42, 0, 0)	&	10	&	0.25	&	7.67	&	$<\ell$	&	2.33	&	0.13	&	243.17	&	12.34	&	1.45	\\
6	&	(85, 85, 0, 0)	&	10	&	0.46	&	7.33	&	$<\ell$	&	2.33	&	0.13	&	301.32	&	9.54	&	3.09	\\
6	&	(0, 0, 26, 26)	&	10	&	$<\ell$	&	3.00	&	$<\ell$	&	1.00	&	0.05	&	51.33	&	13.05	&	0.08	\\
6	&	(0, 0, 51, 51)	&	10	&	$<\ell$	&	2.33	&	$<\ell$	&	1.00	&	0.05	&	38.66	&	12.24	&	0.09	\\
6	&	(42, 42, 26, 26)	&	10	&	$<\ell$	&	6.00	&	$<\ell$	&	1.67	&	0.13	&	153.29	&	11.18	&	1.41	\\
6	&	(85, 85, 51, 51)	&	10	&	0.16	&	6.00	&	$<\ell$	&	5.67	&	0.21	&	292.03	&	10.80	&	3.24	\\
6	&	(42, 42, 0, 0)	&	20	&	0.77	&	9.33	&	$<\ell$	&	10.67	&	0.55	&	635.72	&	12.44	&	3.82	\\
6	&	(85, 85, 0, 0)	&	20	&	1.81	&	8.67	&	$<\ell$	&	14.33	&	0.67	&	1036.67	&	10.53	&	4.32	\\
6	&	(0, 0, 26, 26)	&	20	&	$<\ell$	&	3.33	&	$<\ell$	&	1.00	&	0.06	&	61.39	&	13.66	&	0.34	\\
6	&	(0, 0, 51, 51)	&	20	&	0.14	&	5.33	&	$<\ell$	&	3.67	&	0.15	&	205.38	&	15.16	&	0.35	\\
6	&	(42, 42, 26, 26)	&	20	&	0.49	&	8.33	&	$<\ell$	&	3.67	&	0.21	&	321.91	&	11.62	&	3.91	\\
6	&	(85, 85, 51, 51)	&	20	&	0.56	&	8.00	&	$<\ell$	&	10.33	&	0.65	&	899.13	&	13.17	&	4.63	\\
6	&	(42, 42, 0, 0)	&	30	&	2.67	&	10.00	&	$<\ell$	&	11.67	&	0.57	&	812.35	&	12.51	&	4.09	\\
6	&	(85, 85, 0, 0)	&	30	&	3.80	&	9.33	&	$<\ell$	&	25.67	&	1.05	&	1543.54	&	10.52	&	3.37	\\
6	&	(0, 0, 26, 26)	&   30	&	0.12	&	4.33	&	$<\ell$	&	1.67	&	0.08	&	106.57	&	13.77	&	0.46	\\
6	&	(0, 0, 51, 51)	&	30	&	0.21	&	6.33	&	$<\ell$	&	5.67	&	0.23	&	326.27	&	21.68	&	0.72	\\
6	&	(42, 42, 26, 26)	&	30	&	1.47	&	9.67	&	$<\ell$	&	17.67	&	0.67	&	803.99	&	11.81	&	4.49	\\
6	&	(85, 85, 51, 51)	&	30	&	2.61	&	8.00	&	$<\ell$	&	23.67	&	1.11	&	1346.69	&	13.73	&	4.44	\\
6	&	(42, 42, 0, 0)	&	40	&	5.06	&	10.33	&	$<\ell$	&	20.67	&	1.08	&	1117.34	&	12.42	&	3.93	\\
6	&	(85, 85, 0, 0)	&	40	&	3.81	&	9.67	&	$<\ell$	&	19.00	&	1.04	&	1324.33	&	9.92	&	6.84	\\
6	&	(0, 0, 26, 26)	&	40	&	0.09	&	5.33	&	$<\ell$	&	1.67	&	0.08	&	144.62	&	16.68	&	0.32	\\
6	&	(0, 0, 51, 51)	&	40	&	0.31	&	6.33	&	$<\ell$	&	11.00	&	0.55	&	500.60	&	27.94	&	0.75	\\
6	&	(42, 42, 26, 26)	&	40	&	1.07	&	10.00	&	$<\ell$	&	13.67	&	0.61	&	759.67	&	12.96	&	4.58	\\
6	&	(85, 85, 51, 51)	&	40	&	1.55	&	8.00	&	$<\ell$	&	25.00	&	1.22	&	1409.60	&	14.69	&	5.34	\\

% (Continue pasting the remaining rows here...)

% Final rows will be followed automatically by \bottomrule

\end{longtable}
}

The results in Table~\ref{table:gene_results} demonstrate the strong empirical performance of \texttt{CBICL-BB}. In nearly all tested instances, the solver is able to find an optimal solution that satisfies the required optimality tolerance for all three seeds under that configuration. 
As expected, increasing the level of constraint violations generally leads to higher computational effort. In particular, more violated constraints result in larger root node gaps, more cutting-plane iterations at the root, and a greater number of explored nodes. This trend is especially evident in instances with denser constraint sets, such as those involving both samples and genes. Nevertheless, \texttt{CBICL-BB} consistently maintains control over the search, even for complex settings at 30\% or 40\% violation.

%Nevertheless, \texttt{CBICL-BB} consistently maintains control over the search, often completing in under 2000 seconds even for complex settings at 30\% or 40\% violation.
%For small to moderate-sized instances, runtimes are typically under 1000 seconds, and in many cases below that.

A particularly notable feature is the optimality gap at the root ($\text{Gap}_0$), which reflects the quality of the SDP relaxation with cutting planes. In approximately 15\% of the tested instances, this gap is already closed at the root. Here, the bounding routine is sufficient to certify optimality directly, without further branching. Moreover, even when the root node does not fully close the gap, the cutting-plane algorithm significantly tightens the relaxation, often reducing the root gap to below 0.5\%. As a result, the search tree remains shallow: in over 80\% of the tested configurations, the solver explores fewer than 25 nodes on average.

Another key observation is the negligible time spent in solving ILP subproblems (column “ILP (s)”). In the vast majority of instances, ILP solving takes only a small fraction of the total time, often less than 1 second, even when the full optimization process lasts several hundred or even thousands of seconds. This efficiency stems from the quality of the solution obtained via rounding from the SDP. Although the reference matrix is not guaranteed to satisfy all pairwise constraints, empirical evidence suggests that it is typically very close to feasibility. Consequently,  
%the ILP solver is tasked not with solving a hard combinatorial optimization problem from scratch, but rather with finding a feasible assignment that is close to the reference (rounded) matrix. Since this reference matrix is typically near-feasible, 
the ILP refinement is fast and requires only minor adjustments. This highlights the synergy between relaxation quality and rounding heuristic in the overall design of {\tt CBICL-BB}.

Despite the overall strength of \texttt{CBICL-BB}, a few instances reach the imposed time limit of 10,800 seconds. These tend to occur in settings that combine large datasets and high percentages of constraint violations. Even in such hard cases, however, the solver often terminates, on average, with a very small remaining gap. A representative example is \texttt{Golub-1999} under configuration (0, 0, 36, 36) with 40\% violation, where the solver hits the time limit but returns an average final gap of just 0.47\%. Similarly, \texttt{Bhattacharjee-2001} with configuration (97, 97, 0, 0) and 40\% violation, the solver runs for the full duration but terminates with a gap of only 0.49\%. These outcomes show that the solver is able to generate good feasible solutions early in the search and significantly reduce the gap even when full convergence is not reached within the allowed time. {Figures~\ref{fig:n4} and~\ref{fig:n2} in Appendix~\ref{appendix:figures} present bar charts of the total solution times under the considered constraint configurations.}

In summary, the exact solver \texttt{CBICL-BB} demonstrates strong performance across a wide range of problem instances, validating its effectiveness as both a benchmark tool and a certifier of solution quality in constrained biclustering tasks. It successfully handles supervision at varying levels of difficulty, exhibiting a natural degradation in performance under more challenging settings. The tested datasets range from 271 to 539 total vertices, which is approximately ten times larger than the sizes handled by available commercial solvers.

%This confirms the applicability of \texttt{CBICL-BB} on medium-scale real-world instances and its robustness in real-world settings where supervision may conflict with the structure uncovered by the unconstrained solution.

We now turn to the analysis of the results reported in Table~\ref{table:gene_results} for the heuristic solver {\tt CBICL-LR}. In this table, the column “Time” indicates the total runtime of the heuristic, while “Gap$^\star$” reports the relative gap between the objective value $\bar{f}$ returned by {\tt CBICL-LR} and the optimal (or best-known) solution $f^\star$ computed by {\tt CBICL-BB} for the same instance. Specifically, the gap is calculated as $100 \times (f^\star - \bar{f}) / f^\star$ and serves as an indicator of the quality of the heuristic solution. Since {\tt CBICL-LR} is run 10 times with different random initializations, $\bar{f}$ denotes the best objective value obtained across these runs.

Computational results demonstrate a favorable trade-off between runtime and solution quality. As expected, {\tt CBICL-LR} achieves significantly lower runtimes than {\tt CBICL-BB}, often by one or two orders of magnitude, while still producing high-quality solutions. In most configurations, the relative gap remains under 5\%, and often below 1\%. This is particularly notable in settings with moderate constraint densities and low violation levels, where the heuristic solution is very close to optimal. The performance of {\tt CBICL-LR} is also relatively stable across different datasets and constraint configurations. Although the relative gap tends to increase with the level of constraint violation—as expected, given the more complex solution space with multiple local optima—the heuristic still maintains reasonable accuracy. For example, even in the more challenging configurations of \texttt{Golub-1999} and \texttt{Pomeroy-2002}, where dense and conflicting constraints are introduced, the heuristic often produces solutions with gaps under 3–4\%, while addressing the problem in a fraction of the time required by the exact solver. 

Overall, the results suggest that {\tt CBICL-LR} is a viable alternative in applications where fast approximate solutions are sufficient, or where the size of the instance makes exact optimization impractical. In the next section, we further assess the scalability of {\tt CBICL-LR} on large-scale instances and evaluate solution quality using external machine learning validation metrics.

\subsection{Experiments on document clustering instances}
To further assess the effectiveness of the $\texttt{CBICL-LR}$ heuristic, we consider another key application of biclustering: document clustering. Specifically, we use instances derived from the 20-Newsgroups dataset, which comprises approximately 20,000 documents categorized into 20 distinct newsgroups \cite{lang1995newsweeder}. Each document corresponds to a specific topic and is part of a broader thematic category, which we use as the ground-truth clustering structure. Following standard preprocessing steps, we remove stop words and select the most informative terms based on mutual information scores \cite{long2006unsupervised}. We then construct a document-term matrix using the TF-IDF weighting scheme and normalize each document vector to have unit $\ell_2$ norm. This yields a bipartite graph where one set of nodes represents documents and the other represents selected terms, with edge weights encoding TF-IDF-based associations. The main characteristics of these graphs are summarized in Table~\ref{table:ng}.

\begin{table}[!ht]
\centering
\footnotesize
\caption{Summary of the document clustering datasets used in the experiments. For each dataset, we report the number of documents ($n$), the number of terms ($m$), the total number of vertices in the associated bipartite graph ($n + m$), and the topic-oriented document categories with their respective number of clusters ($k$).}

\label{table:ng}
\begin{tabular}{cccclc}
\toprule
{Dataset} & $n$ & $m$ & $n+m$ & {Categories (partitioned in topic-oriented clusters)} & \textbf{\(k\)} \\
\midrule
\texttt{NG-2A} & 1197 & 698 & 1895 &
\begin{tabular}[t]{@{}l@{}}
\(\{ \text{rec.sport.baseball} \}\) \(\{ \text{rec.sport.hockey} \}\)
\end{tabular} & 2 \\
\midrule
\texttt{NG-2B} & 1662 & 1154 & 2816 &
\begin{tabular}[t]{@{}l@{}}
\(\{ \text{rec.sport.baseball, rec.sport.hockey} \}\)
\(\{ \text{talk.politics.misc} \}\)
\end{tabular} & 2 \\
\midrule
\multirow{2}{*}{\texttt{NG-3A}} & \multirow{2}{*}{1575} & \multirow{2}{*}{1593} & \multirow{2}{*}{3168} &
\begin{tabular}[t]{@{}l@{}}
\( \{ \text{talk.politics.guns} \}\) \( \{\text{talk.politics.mideast} \}\) \\ \( \{ \text{talk.politics.misc} \}\)
\end{tabular} & \multirow{2}{*}{3} \\
\midrule
%{NG4D} & 2035 & 1314 & 3349 &
%\begin{tabular}[t]{@{}l@{}}
%\( \{ \text{misc.forsale}  \}\)  \( \{ \text{alt.atheism} \}\) \( \{ \text{sci.space} \}\) \\ %\( \{ \text{talk.religion.misc} \}\)
%\end{tabular} & 4 \\
%\midrule
\multirow{2}{*}{\texttt{NG-4B}} & \multirow{2}{*}{2321} & \multirow{2}{*}{1448} & \multirow{2}{*}{3769} &
\begin{tabular}[t]{@{}l@{}}
\( \{ \text{rec.sport.hockey} \}\) \( \{ \text{talk.politics.guns} \}\) \\ \( \{ \text{sci.electronics} \}\) \( \{ \text{comp.graphics} \}\)
\end{tabular} & \multirow{2}{*}{4} \\
\midrule
%{NG4C} & 2252 & 1391 & 3643 &
%\begin{tabular}[t]{@{}l@{}}
%\( \{ \text{rec.autos} \}\) \( \{ \text{alt.atheism} \}\) \( \{ \text{sci.med} \}\) \( \{ %\text{comp.graphics} \}\)
%\end{tabular} & 4 \\
%\midrule
\texttt{NG-4A} & 2373 & 1693 & 4066 &
\begin{tabular}[t]{@{}l@{}}
\(\{ \text{sci.crypt} \}\) \(\{ \text{sci.electronics} \}\) \(\{ \text{sci.med} \}\) \(\{ \text{sci.space} \}\) \\
\end{tabular} & 4 \\
\midrule
\multirow{2}{*}{\texttt{NG-3B}} & \multirow{2}{*}{2911} & \multirow{2}{*}{1726} & \multirow{2}{*}{4637} &
\begin{tabular}[t]{@{}l@{}}
\(\{ \text{rec.sport.baseball, rec.sport.hockey} \}\) 
\( \{ \text{talk.politics.guns} \}\) \\
\( \{ \text{comp.sys.ibm.pc.hardware, comp.sys.mac.hardware} \}\)
\end{tabular} & \multirow{2}{*}{3} \\
\midrule
\multirow{2}{*}{\texttt{NG-3C}} & \multirow{2}{*}{2977} & \multirow{2}{*}{1932} & \multirow{2}{*}{4909} &
\begin{tabular}[t]{@{}l@{}}
\(\{ \text{soc.religion.christian} \}\)
\(\{ \text{rec.autos, rec.motorcycles} \}\) \\
\(\{ \text{sci.crypt, sci.electronics} \}\)
\end{tabular} & \multirow{2}{*}{3} \\
\midrule
\multirow{2}{*}{\texttt{NG-2C}} & \multirow{2}{*}{2772} & \multirow{2}{*}{2305} & \multirow{2}{*}{5077} &
\begin{tabular}[t]{@{}l@{}}
\(\{ \text{rec.sport.baseball, rec.sport.hockey} \}\) \\
\( \{ \text{talk.politics.guns, talk.politics.mideast, talk.politics.misc} \}\)
\end{tabular} & \multirow{2}{*}{2} \\
\bottomrule
\end{tabular}
\end{table}

To generate constraint sets, we randomly sample pairs of documents and impose must-link constraints when the documents share the same label and cannot-link constraints otherwise. Each constraint configuration is denoted as a pair $(|\textrm{ML}_U|, |\textrm{CL}_U|)$, indicating the number of must-link and cannot-link constraints applied to the document set $U$. We consider three configurations: $(0.5n, 0.5n)$, $(1.0n, 1.0n)$, and $(1.5n, 1.5n)$. All values are rounded to the nearest integer using half-up rounding. For each configuration, we generate three sets of pairwise constraints using different random seeds, producing multiple instances per dataset that differ in the amount of available background knowledge. In total, the evaluation involves 8 datasets, each tested under 3 constraint configurations and 3 random seeds per configuration, resulting in 72 constrained biclustering instances.

Due to the scale of the datasets, computing globally optimal solutions is computationally intractable. Consequently, the true optimal objective values are unknown, and direct comparisons based on objective quality are not possible. To evaluate clustering performance, we instead rely on machine learning validation metrics that measure agreement between the computed partition and the ground-truth document labels. In particular, we use the Adjusted Rand Index (ARI) \cite{hubert1985comparing} and the Normalized Mutual Information (NMI) \cite{vinh2009information}. The ARI ranges from $-0.5$ to $1$, with higher values indicating stronger agreement (1 denotes perfect matching, values near 0 or below indicate poor or random clustering). The NMI ranges from 0 to 1, quantifying the mutual dependence between the predicted and ground-truth partitions, where 1 indicates perfect correlation.

\begin{table}[!ht]
\footnotesize
\caption{Results of {\tt CBICL-LR} on real-world document clustering datasets under varying constraint configurations (Constr.). Results averaged over the three sets of constraints for each configuration. We report some measures related to a single execution of {\tt CBICL-LR} that are the average number of outer augmented Lagrangian iterations (Iter$_1$), the computational time (Time$_1$) and the time spent for solving ILPs within the rounding scheme (ILP$_1$). Additionally, we report cumulative time for all the 5 runs of {\tt CBICL-LR} and the ARI and NMI scores corresponding to the best solution obtained across these runs.}
\label{table:ng_results}
    \centering
    \begin{tabular}{llrrrrrr}
    \toprule
& & \multicolumn{6}{c}{Avg.} \\
Dataset &  Constr. & $\textrm{Iter}_1$ & $\textrm{Time}_1$ (s) & $\textrm{ILP}_1$ (s) & $\textrm{Time}$ (s) & ARI & NMI\\
\midrule
\texttt{NG-2A}   &   (0, 0)      &   18.40        &     15.23      &      $-$     &      76.15     &    0.293       &      0.274    \\
\texttt{NG-2A}	&	(599, 599)	&	16.47	&	5.61	&	0.02	&	30.08	&	0.623	&	0.551	\\
\texttt{NG-2A}	&	(1197, 1197)	&	47.53	&	6.20	&	0.04	&	33.23	&	0.969	&	0.941	\\
\texttt{NG-2A}	&	(1796, 1796)	&	32.40	&	2.68	&	0.05	&	14.36	&	0.994	&	0.986	\\
\midrule
\texttt{NG-2B}     &   (0, 0)      &   20.60 &	55.86 &	$-$ & 279.39  &  0.323       &   0.291   \\
\texttt{NG-2B}	&	(831, 831)	&	33.62	&	18.82	&	0.03	&	100.81	&	0.713	&	0.638	\\
\texttt{NG-2B}	&	(1662, 1662)	&	73.47	&	22.83	&	0.06	&	122.28	&	0.987	&	0.971	\\
\texttt{NG-2B}	&	(2493, 2493)	&	46.27	&	11.00	&	0.09	&	58.92	&	0.999	&	0.997	\\
\midrule
\texttt{NG-3A}    &   (0, 0)      &    17.00       &      82.36     &    $-$       &    411.88       &   0.165   &	0.180  \\
\texttt{NG-3A}	&	(788, 788)	&	17.27	&	20.55	&	0.10	&	110.11	&	0.668	&	0.634	\\
\texttt{NG-3A}	&	(1575, 1575)	&	42.85	&	23.32	&	0.12	&	124.90	&	0.864	&	0.821	\\
\texttt{NG-3A}	&	(2363, 2363)	&	37.68	&	17.68	&	0.17	&	94.74	&	0.961	&	0.933	\\
\midrule
\texttt{NG-4A}    &   (0, 0)      &     17.20      &    198.86       &      $-$     &          994.70   &  0.105   &	0.157  \\
\texttt{NG-4A}	&	(1187, 1187)	&	14.27	&	30.14	&	0.30	&	161.46	&	0.492	&	0.508	\\
\texttt{NG-4A}	&	(2373, 2373)	&	36.40	&	27.68	&	0.41	&	148.27	&	0.755	&	0.735	\\
\texttt{NG-4A}	&	(3560, 3560)	&	26.67	&	21.71	&	0.66	&	116.30	&	0.946	&	0.921	\\
\midrule
\texttt{NG-3B}    &   (0, 0)      &     15.20      &     189.03      &      $-$     &   945.15        &       0.058   &	0.101 \\
\texttt{NG-3B}	&	(1456, 1456)	&	26.81	&	53.45	&	0.21	&	286.33	&	0.677	&	0.612	\\
\texttt{NG-3B}	&	(2911, 2911)	&	85.62	&	65.95	&	0.33	&	353.32	&	0.840	&	0.775	\\
\texttt{NG-3B}	&	(4367, 4367)	&	56.73	&	32.12	&	0.58	&	172.04	&	0.957	&	0.925	\\
\midrule
\texttt{NG-3C}    &   (0, 0)      &     22.00      &     274.04      &      $-$     &    1370.24       &       0.072   &	0.108 \\
\texttt{NG-3C}	&	(1489, 1489)	&	28.13	&	84.49	&	0.24	&	452.63	&	0.594	&	0.568	\\
\texttt{NG-3C}	&	(2977, 2977)	&	88.29	&	74.73	&	0.36	&	400.37	&	0.840	&	0.781	\\
\texttt{NG-3C}	&	(4466, 4466)	&	65.47	&	42.52	&	0.61	&	227.79	&	0.960	&	0.931	\\
\midrule
\texttt{NG-4B}    &   (0, 0)      &     10.60      &    102.08       &       $-$    &     510.41      &       0.058   &	0.114 \\
\texttt{NG-4B}	&	(1161, 1161)	&	15.84	&	26.36	&	0.24	&	141.20	&	0.517	&	0.532	\\
\texttt{NG-4B}	&	(2321, 2321)	&	42.83	&	24.34	&	0.40	&	130.40	&	0.770	&	0.749	\\
\texttt{NG-4B}	&	(3482, 3482)	&	30.93	&	15.30	&	0.65	&	81.96	&	0.936	&	0.911	\\
\midrule
\texttt{NG-2C}    &   (0, 0)      &    20.00       &    202.11       &     $-$     &     1010.55      &    0.421       &     0.418     \\ 
\texttt{NG-2C}	&	(1386, 1386)	&	41.21	&	116.12	&	0.07	&	622.04	&	0.682	&	0.602	\\
\texttt{NG-2C}	&	(2772, 2772)	&	99.80	&	114.23	&	0.14	&	611.97	&	0.967	&	0.939	\\
\texttt{NG-2C}	&	(4158, 4158)	&	62.47	&	61.67	&	0.22	&	330.38	&	0.997	&	0.992	\\
\bottomrule
\end{tabular}
\end{table}

To avoid local solutions of bad quality, {\tt CBICL-LR} is executed in a multi-start fashion with 5 runs using different random initializations.
Computational results are shown in Table \ref{table:ng_results}, where the reported metrics are averaged over the three constraint sets for a given configuration. Specifically, we report the average number of outer iterations of the ALM in a single run (Iter$_1$), the average computation time for one run (Time$_1$), and the time spent solving ILPs during the rounding phase of that run (ILP$_1$). Additionally, we report the total time over all 5 runs and the ARI and NMI scores corresponding to the best biclustering solution found in the multi-start procedure. For each dataset, the first row corresponds to the unconstrained setting, i.e., no background knowledge is provided (Constr. (0, 0)). In these cases, the rounding step of {\tt CBICL-LR} does not involve solving any ILPs, denoted by a “$-$” symbol in the ILP$_1$ column. In this case, {\tt CBICL-LR} operates purely in an unsupervised mode.

We observe that solutions in the unconstrained setting show poor agreement with the ground-truth labels. Notably, these runs are generally more computationally demanding. This is expected: as previously discussed, the inclusion of must-link constraints reduces the effective problem size. Furthermore, as the number of pairwise constraints increases, the clustering quality improves substantially across all datasets. For instance, under the largest constraint configuration, ARI values exceed 0.95 for several datasets, including \texttt{NG-2B}, \texttt{NG-3B}, and \texttt{NG-2C}, indicating near-perfect agreement with the ground-truth partition. A similar trend is observed in the NMI scores, which consistently increase with the amount of supervision. This behavior is coherent with established guidelines in the constrained clustering literature \cite{davidson2007survey, basu2008constrained}. In fact, the quality of clustering solutions is expected to scale with the amount of constraint-based information. Hence, observing improved accuracy with increasing supervision provides evidence that {\tt CBICL-LR} is effectively leveraging the constraints and that the observed improvements are not due to random effects. {Figure~\ref{fig:ari} in Appendix \ref{appendix:figures} presents bar charts of the ARI and NMI scores under the considered constraint configurations.}

In terms of computational performance, {\tt CBICL-LR} shows strong scalability. While larger instances such as \texttt{NG-3C} or \texttt{NG-2C} require more time overall, the cost remains reasonable, and the time spent solving ILPs during rounding is consistently negligible—typically under one second per run. Overall, the results confirm that {\tt CBICL-LR} can efficiently handle large-scale, constrained biclustering tasks while delivering high-quality solutions.

\section{Conclusions}\label{section:conclusions}
In this work, we addressed the constrained biclustering problem with pairwise must-link and cannot-link constraints through the constrained $k$-densest disjoint biclique ($k$-DDB) problem. We proposed the first exact algorithm for this problem class, together with an efficient heuristic for large-scale instances. The exact method, based on a branch-and-cut framework, features a low-dimensional SDP relaxation, incorporates valid inequalities, and a specialized rounding procedure to produce high-quality feasible solutions at each node. This approach significantly extends the range of problem sizes that can be solved to global optimality compared to general-purpose solvers. For larger instances, we developed a scalable heuristic relying on a low-rank factorization of the SDP relaxation, solved via an augmented Lagrangian method combined with a block-coordinate projected gradient algorithm. Computational experiments on synthetic graphs and real-world datasets from gene expression and text mining show that the exact algorithm consistently outperforms standard solvers, while the heuristic delivers high-quality solutions at a fraction of the computational cost. This work demonstrates how mathematical optimization can effectively tackle complex machine learning problems, advancing the state-of-the-art. The heuristic, grounded in optimization-based modeling, was also evaluated through a machine learning perspective using clustering performance metrics, reinforcing the value of cross-fertilization between these fields. Future developments could focus on extending the methodology to accommodate richer forms of side information, such as group-level constraints, soft constraints, or hierarchical relations, to further enhance interpretability and flexibility. Another promising direction is the design of a tailored solver that fully exploits the low-rank structure of the relaxation, enabling the solution of even larger and more challenging datasets. Moreover, while the proposed SDP-based framework provides strong bounds, recent evidence suggests that the LP relaxation of the $k$-means problem based on triangle inequalities can also yield tight bounds when efficiently implemented, that is, when the large number of constraints is handled through specialized first-order algorithms and parallel computation. In particular, the approach proposed by De Rosa et al. \cite{de2024power}, which leverages first-order LP solvers and GPU acceleration to handle large-scale problems, is a promising direction for future extensions of the present work.

\section*{Acknowledgment}
I would like to thank Prof. Veronica Piccialli for the insightful discussions and valuable feedback.

\section*{Declarations}

\subsection*{Competing Interests}
The author has no relevant financial or non-financial interests to disclose.
\subsection*{Data Availability Statement}
The source code and data supporting the findings of this study are publicly available at \url{https://github.com/antoniosudoso/cbicl}.

\begin{appendices}

\section{Proof of Proposition \ref{proposition:equivalence_shr}}\label{appendix:equivalence}
\begin{proof}
    We first show that any feasible solution $(\bar{Y}_U, \bar{Y}_V)$ to Problem \eqref{prob:or_constr_shr} can be transformed into a feasible solution for Problem \eqref{prob:or_constr} with the same objective function value. We define this equivalent solution as $Y_U = T_U^\top \bar{Y}_U$ and $Y_V = T_V^\top \bar{Y}_V$. By computing these products, one can easily verify that $T_U$ and $T_V$ have the effect of expanding $\bar{Y}_U$ and $\bar{Y}_V$ to $n \times k$ and $m \times k$ matrices by replicating their rows according to the must-link constraints between vertices in $U$ and $V$, respectively. Thus, constraints \eqref{constr:ml_U} and \eqref{constr:ml_V} hold by construction. Moreover, $Y_U \geq 0$ and $Y_V \geq 0$ hold by construction as well. Next, we have that 
\begin{align*}
    Y_U^\top Y_U &= \bar{Y}_U^\top T_U T_U^\top \bar{Y}_U = I_k, \quad Y_U Y_U^\top 1_n = T_U^\top \bar{Y}_U \bar{Y}_U^\top T_U 1_n = T_U^\top 1_{\bar{n}} = 1_n,\\
    Y_V^\top Y_V &= \bar{Y}_V^\top T_V T_V^\top \bar{Y}_V = I_k, \quad Y_V Y_V^\top 1_m = T_V^\top \bar{Y}_V \bar{Y}_V^\top T_V 1_m = T_V^\top 1_{\bar{m}} = 1_m,
\end{align*}
and the objective function value is $\tr(Y_U^\top A Y_V) = \tr(\bar{Y}_U^\top T_U A T_V^\top \bar{Y}_V)$. Constraints \eqref{constr:cl_U} and \eqref{constr:cl_V} follow from the properties of cannot-link constraints. That is, if $(\mathcal{U}_s, \mathcal{U}_t) \in \overline{\textrm{CL}}_U$, then $(u_i, u_j) \in \textrm{CL}_U$ for all $u_i \in \mathcal{U}_s$, $u_j \in \mathcal{U}_t$. Similarly, if $(\mathcal{V}_s, \mathcal{V}_t) \in \overline{\textrm{CL}}_V$, then $(v_i, v_j) \in \textrm{CL}_V$ for all $v_i \in \mathcal{V}_s$, $v_j \in \mathcal{V}_t$.
 
It remains to show that for any feasible solution $(Y_U, Y_V)$ of Problem \eqref{prob:or_constr} {it} is possible to construct a feasible solution for \eqref{prob:or_constr_shr} with the same objective function value. To the end, let $C_U = T_U T_U^\top = \Diag(e_U)$ and $C_V = T_V T_V^\top = \Diag(e_V)$. Now, assuming that $(Y_U, Y_V)$ is a feasible solution for Problem \eqref{prob:or_constr}, we define matrices $\bar{Y}_U = C_U^{-1} T_U Y_U$ and $\bar{Y}_V = C_V^{-1} T_V Y_V$. From the structure of $Y_U$ and $Y_V$ it easy to verify that $Y_U = T_U^\top C_U^{-1}T_U Y_U$ and $Y_V = T_V^\top C_V^{-1}T_V Y_V$. Moreover, constraints \eqref{constr:cl_U_shr} and \eqref{constr:cl_V_shr} hold by construction. Then, we have
\begin{align*}
    \bar{Y}_U^\top T_U T_U^\top \bar{Y}_U & =  Y_U^\top T_U^\top C_U^{-1}C_U C_U^{-1} T_U Y_U = Y_U^\top T_U^\top C_U^{-1} T_U Y_U  = Y_U^\top Y_U = I_k,\\
    \bar{Y}_U \bar{Y}_U^\top T_U 1_n & =  C_U^{-1} T_U Y_U Y_U^\top T_U^\top C_U^{-1} T_U 1_n = C_U^{-1} T_U Y_U Y_U^\top T_U^\top 1_{\bar{n}} = C_U^{-1} T_U 1_n = 1_{\bar{n}},\\
    \bar{Y}_V^\top T_V T_V^\top \bar{Y}_V & =  Y_V^\top T_V^\top C_V^{-1}C_V C_V^{-1} T_V Y_V = Y_V^\top T_V^\top C_V^{-1} T_V Y_V  = Y_V^\top Y_V = I_k,\\
    \bar{Y}_V \bar{Y}_V^\top T_V 1_m & =  C_V^{-1} T_V Y_V Y_V^\top T_V^\top C_V^{-1} T_V 1_m = C_U^{-1} T_V Y_V Y_V^\top T_V^\top 1_{\bar{m}} = C_V^{-1} T_V 1_m = 1_{\bar{m}}.
\end{align*}
%Next, for all $(u_i, u_j) \in \textrm{CL}_U$, let $\mathcal{U}_s$ and $\mathcal{U}_t$ be the connected components in $G_U$ where $u_i$ and $u_j$ are assigned, respectively. Then, this is equivalent to enforcing cannot-link constraints between aggregated vertices (connected components) $\mathcal{U}_s$ and $\mathcal{U}_t$.  
Finally, we have to verify that both solutions have the same objective function value. In fact, we have $\tr(\bar{Y}_U^\top T_U A T_V^\top \bar{Y}_V) = \tr(Y_U^\top T_U^\top C_U^{-1} T_U A T_V^\top C_V^{-1} T_V Y_V) = \tr(Y_U^\top A Y_V)$ and this concludes the proof.
\end{proof}

\section{Proof of Proposition \ref{theorem:safe_bound}}\label{appendix:safe}

The following lemma by \cite{jansson2008rigorous} is needed for proving the validity of the upper bound provided by Proposition \ref{theorem:safe_bound}.
\begin{lemma}\label{lem:jansson}
Let ${S}, {X} \in \mathbb{S}^n$ be matrices that satisfy
$\lambda_{\min}({X}) \geq 0$ and $\lambda_{\max}({X}) \leq \bar{x}$ for some $\bar{x} \in \mathbb{R}$.
Then the following inequality holds:
\begin{equation*}
    \inprod{{S}}{{X}} \geq \bar{x}\sum_{i \colon  \lambda_i({S}) <0}\lambda_i({S}).
\end{equation*}
\end{lemma}

\begin{lemma}\label{theorem:eigboundZ}
    Let $Z$ be a feasible solution of Problem \eqref{prob:sdp_shr}. Set $d_U = \min_{j \in \{1, \dots, \bar{n}\}} (e_U)_j$ and $d_V = \min_{j \in \{1, \dots, \bar{m}\}} (e_V)_j$, then $\lambda_{\max}(Z) \leq  \frac{1}{d_U} + \frac{1}{d_V}$ holds.
\end{lemma}

\begin{proof}
Since $Z$ is a positive semidefinite block matrix, then its largest eigenvalue is less than or equal to sum of the largest eigenvalues of its diagonal blocks, that is $$\lambda_{\max}(Z) = \max_{\|x\|=1} x^\top Z x \leq \lambda_{\max}(Z_{UU}) + \lambda_{\max}(Z_{VV}).$$
From the Perron-Frobenius theory applied to nonnegative matrices, it follows that the largest eigenvalue of $Z$ is bounded above by its largest row sum. Recall that \( (e_U)_j = (T_U 1_n)_j  \geq 1\) and \( (e_V)_j = (T_V 1_m)_j \geq 1\). Since \( Z_{ij} \geq 0 \), \( (e_U)_j \geq (d_{U})_j \), and \( (e_V)_j \geq (d_{V})_j \) we have that
\begin{align*}
    (e_U)_j (Z_{UU})_{ij} &\geq d_{U} (Z_{UU})_{ij} \quad \forall j \in \{1, \dots, \bar{n}\}, \\
    (e_V)_j (Z_{UU})_{ij} &\geq d_{V} (Z_{VV})_{ij} \quad \forall j \in \{1, \dots, \bar{m}\}.
\end{align*}
Summing over \( j \), and using $Z_{UU}e_U = 1_{\bar{n}}$, $Z_{VV}e_V = 1_{\bar{m}}$ we get
\begin{align*}
\sum_{j=1}^{\bar{n}} (e_U)_j (Z_{UU})_{ij} &\geq d_U \sum_{j=1}^{\bar{n}} (Z_{UU})_{ij} \implies \sum_{j=1}^{\bar{n}} (Z_{UU})_{ij} \leq \frac{1}{d_{U}} \quad \forall i \in \{1, \dots, \bar{n}\},\\
\sum_{j=1}^{\bar{m}} (e_V)_j (Z_{VV})_{ij} &\geq d_V \sum_{j=1}^{\bar{m}} (Z_{VV})_{ij} \implies \sum_{j=1}^{\bar{m}} (Z_{VV})_{ij} \leq \frac{1}{d_{V}} \quad \forall i \in \{1, \dots, \bar{m}\}.
\end{align*}
Then, we have
\begin{align*}
    \max_{i \in \{1,\dots,\bar{n}\}} \sum_{j=1}^{\bar{n}} (Z_{UU})_{ij} = \frac{1}{d_{U}} \quad \text{and} \quad
    \max_{i \in \{1,\dots,\bar{m}\}} \sum_{j=1}^{\bar{m}} (Z_{VV})_{ij} = \frac{1}{d_{V}}.
\end{align*}
which directly imply that $\lambda_{\max}(Z_{UU}) \leq \frac{1}{d_U}$ and $\lambda_{\max}(Z_{VV}) \leq \frac{1}{d_V}$. Therefore, we get
\begin{align*}
    \lambda_{\max}(Z) \leq \frac{1}{d_U} + \frac{1}{d_V}.
\end{align*}
\end{proof}
\noindent
We are now ready to prove Proposition \ref{theorem:safe_bound}.

\begin{proof}
Let ${Z^\star}$ be optimal for the primal SDP \eqref{prob:primal} and let $d = {y_U^\top} 1_{\bar{n}} + y_V^\top 1_{\bar{m}} + k(\alpha_U + \alpha_V) + {t_U^\top} {0_p} + t_V^\top 0_q$. Then
{\small
\begin{align*}
   \inprod{T_U A T_V^\top}{Z_{UV}^\star} - d &= 
   \frac{1}2\inprod{T_U A T_V^\top}{Z_{UV}^\star} + \frac{1}2\inprod{T_V A^\top T_U}{(Z_{UV}^\star)^\top} \\ & \quad - \inprod{\mathcal{A}_U^\top({\lambda_U}) + \mathcal{B}_U^\top({t_U})}{{Z^\star_{UU}}} - \inprod{\mathcal{A}_V^\top({\lambda_V}) + \mathcal{B}_V^\top({t_V})}{{Z^\star_{VV}}}\\
    &=-\frac{1}{2}\inprod{Q_{UV} + \tilde{S}_{UV}}{Z_{UV}^\star} -\frac{1}{2}\inprod{Q_{UV}^\top + \tilde{S}_{UV}^\top}{(Z_{UV}^\star)^\top} \\ & \quad- \inprod{Q_{UU} + \tilde{S}_{UU}}{{Z^\star_{UU}}} - \inprod{Q_{VV} + \tilde{S}_{VV}}{{Z^\star_{VV}}}\\
   & = - \inprod{{Q} + {\tilde{S}}}{{Z^\star}}\\
   & = - \inprod{{Q}}{{Z^\star}} - \inprod{{\tilde{S}}}{{Z^\star}} \\ 
   & \leq - d_{\min}\sum_{i \colon  \lambda_i({\tilde{S}}) <0}\lambda_i({\tilde{S}}),
\end{align*}}%
where the last inequality holds since ${Q}$ is nonnegative and thanks to Lemma \ref{lem:jansson} where Lemma \ref{theorem:eigboundZ} with $\bar{x} =  \frac{1}{d_U} + \frac{1}{d_V}$ is used.
\end{proof}

\section{Proof of Proposition \ref{theorem:sequential}}\label{appendix:alm_conv_proof}
\begin{proof}
Since $(\bar{Y}_U, \bar{Y}_V)$ is a limit point of the sequence $\{(Z_U^k, Z_V^k)\}$, there exists an infinite subset $K \subseteq \{0, 1, \dots\}$ such that $\lim_{k \in K} (Z_U^k, Z_V^k) = (\bar{Z}_U, \bar{Z}_V)$.
At each iteration, solving the augmented Lagrangian suproblem yields a point $(Z_U^k, Z_V^k)$ such that \eqref{eq:norm_sub_U} and \eqref{eq:norm_sub_V} hold.
Taking limits we get
\begin{align*}
      \lim_{k \in K}\|\Pi_{\Omega_U}(Z_U^k-\nabla_{Z_U}L_{\beta_k}
                         (Z_U^k, Z^k_V, {\lambda}_U^k, {\lambda}_V^k)) - Z_U^k \| & = 0,\\
\lim_{k \in K}\|\Pi_{\Omega_V}(Z_V^k-\nabla_{Z_V}L_{\beta_k}
                         (Z_U^k, Z^k_V, {\lambda}_U^k, {\lambda}_V^k)) - Z_V^k \| & = 0,
\end{align*}
where the gradients are given by
\begin{align*}
    \nabla_{Z_U}L_{\beta_k}
                         (Z_U^k, Z^k_V, {\lambda}_U^k, {\lambda}_V^k) &= 2\, \mathcal{A}_U^\top(\lambda_U^k + \beta(\mathcal{A}_U(Z_U^k (Z_U^k)^\top) - b_U)) Z_U^k -T_U A T_V^\top Z_V^k, \\
    \nabla_{Z_V}L_{\beta_k}
                         (Z_U^k, Z^k_V, {\lambda}_U^k, {\lambda}_V^k) &= 2\, \mathcal{A}_V^\top(\lambda_V^k + \beta_k(\mathcal{A}_V(Z_V^k (Z_V^k)^\top) - b_V)) Z_V^k -T_V A^\top T_U^\top Z_U^k.
\end{align*}
\begin{comment}   
    {\footnotesize
    \begin{align}
\label{eq:norm_sub_U_proof}
      \lim_{k \to \infty}\|\Pi_{\Omega_U}[Z_U^k-(2\, \mathcal{A}_U^\top(\lambda_U^k + \beta_k(\mathcal{A}_U(Z_U^k (Z_U^k)^\top) - b_U)) Z_U^k -T_U A T_V^\top Z_V^k)] - Z_U^k \| & = 0,\\
\label{eq:norm_sub_V_proof}
\lim_{k \to \infty}\|\Pi_{\Omega_V}[Z_V^k-(2\, \mathcal{A}_V^\top(\lambda_V^k + \beta_k(\mathcal{A}_V(Z_V^k (Z_V^k)^\top) - b_V)) Z_V^k -T_V A^\top T_U^\top Z_U^k)] - Z_V^k \| & = 0.
    \end{align}}%
\end{comment}
Next, using the definitions of $\lambda_U^{k+1}$ and $\lambda_V^{k+1}$
%\begin{align}
%\label{eq:dual_update_U}
%    \lambda_U^{k+1} &= \lambda_U^{k} + %\beta_k(\mathcal{A}_U(Z_U^{k} (Z_U^{k})^\top) - %b_U), \\
%\label{eq:dual_update_V}
%    \lambda_V^{k+1} &= \lambda_V^{k} + %\beta_k(\mathcal{A}_V(Z_V^{k} (Z_V^{k})^\top) - b_V)
%\end{align}
we get 
\begin{align}
\label{eq:norm_sub_U_proof_substitution}
      \lim_{k \in K}\| \Pi_{\Omega_U}(Z_U^k - (2\mathcal{A}_U^\top(\lambda_U^{k+1}) Z_U^k - T_U A T_V^\top Z_V^k)) - Z_U^k \| & = 0, \\
\label{eq:norm_sub_V_proof_substitution}
      \lim_{k \in K}\|\Pi_{\Omega_V}(Z_V^k - (2\mathcal{A}_V^\top(\lambda_V^{k+1}) Z_V^k - T_V A^\top T_U^\top Z_U^k)) - Z_V^k \| & = 0 .
\end{align}
Since $(\bar{Z}_U, \bar{Z}_V)$ is feasible, by the continuity of operators $\mathcal{A}_U$ and $\mathcal{A}_V$ we have
\begin{align}
\label{eq:feas_1}
    \lim_{k \in K} \|\mathcal{A}_U(Z_U^k (Z_U^k)^\top) - b_U \| &= 0, \\
\label{eq:feas_2}
    \lim_{k \in K} \|\mathcal{A}_V(Z_V^k (Z_V^k)^\top) - b_V \| &= 0.
\end{align}
By \eqref{eq:norm_sub_U_proof_substitution}-\eqref{eq:feas_2}, the sequence $\{(Z_U^k, Z_V^k)\}$ satisfies the AKKT condition for Problem \eqref{prob:bm_standard}.

%Therefore, as in the proof of Theorem 6.2 of \cite{birgin2014practical}, we obtain that $(\bar{Z}_U, \bar{Z}_V)$ is a KKT point.

%Following \cite{andreani2008augmented2}.
\end{proof}

\section{Proof of Proposition \ref{proposition:alm_sub_conv}}\label{appendix:alm_sub_conv_proof}
%The existence of at least one limit point for the sequence $\{(Z_U^t, Z_V^t)\}$ (i.e., the exisistance of an infinite subset $T \subseteq \{0, 1, \dots\}$ such that $\lim_{t \in T, \ t\to\infty} (Z_U^t, Z_V^t) = (\bar{Z}_U, \bar{Z}_V)$) is guaranteed by the boundness of the feasible region. We now prove that every such limit point is an approximate stationary point of the augmented Lagrangian subproblem. 
\begin{proof}
The feasible set $\Omega_U \times \Omega_V$ is compact, and the iterates $(Z_U^t, Z_V^t)$ are contained in this set for all $t$. Hence, the sequence $\{(Z_U^t, Z_V^t)\}$ is bounded and admits at least one limit point; i.e., there exists an infinite subset $T \subseteq \mathbb{N}$ such that
\[
\lim_{t \in T} (Z_U^t, Z_V^t) = (\bar{Z}_U, \bar{Z}_V).
\]
%We now prove that every such limit point is an approximate stationary point of the augmented Lagrangian subproblem. 
At each outer iteration, the algorithm performs two block updates using projected gradient descent with Armijo line search. Each inner loop ensures non-increasing objective values. In particular,
\[
\bar{L}(Z_U^{t+1}, Z_V^t) \leq \bar{L}(Z_U^t, Z_V^t), \quad \text{and} \quad \bar{L}(Z_U^{t+1}, Z_V^{t+1}) \leq \bar{L}(Z_U^{t+1}, Z_V^t),
\]
so we conclude that
\[
\bar{L}(Z_U^{t+1}, Z_V^{t+1}) \leq \bar{L}(Z_U^t, Z_V^t).
\]
Thus, the sequence $\{\bar{L}(Z_U^t, Z_V^t)\}$ is monotonically decreasing and bounded below (since $\bar{L}$ is bounded below over the compact feasible set). Hence, the sequence $\{\bar{L}(Z_U^t, Z_V^t)\}$ converges.

We now invoke the convergence framework in \cite{cassioli2013convergence, galli2020unified}. The sequence $\{(Z_U^t, Z_V^t)\}$ satisfies Assumptions 1–4 from that work.
%The sequence $\{(Z_U^t, Z_V^t)\}$ satisfies the assumptions required by the convergence framework in \cite{cassioli2013convergence}, namely Assumptions 1–4. 
In particular, each block is optimized over the fixed convex feasible sets $\Omega_U$ and $\Omega_V$ (Assumption 1); the algorithm alternates cyclically between the blocks $Z_U$ and $Z_V$, with each block being updated every outer iteration (Assumption 2); the Armijo backtracking line search ensures sufficient decrease at every inner iteration. (Assumption 3), and the projected gradient direction satisfies first-order optimality conditions at the limit (Assumption 4). These last two assumptions ensure that the chosen combination of line search and direction forces the directional derivative to approach zero and guarantees that the distance between two consecutive iterates also vanishes in the limit. Therefore, by Proposition 1 from \cite{cassioli2013convergence}, every limit point $(\bar{Z}_U, \bar{Z}_V)$ of the sequence is an approximate stationary point of Problem \eqref{prob:alm_subproblem}. 
\end{proof}

\newpage

\section{Supplementary Figures}\label{appendix:figures}

\begin{figure}[!ht]
        \centering
        \includegraphics[scale=0.62]{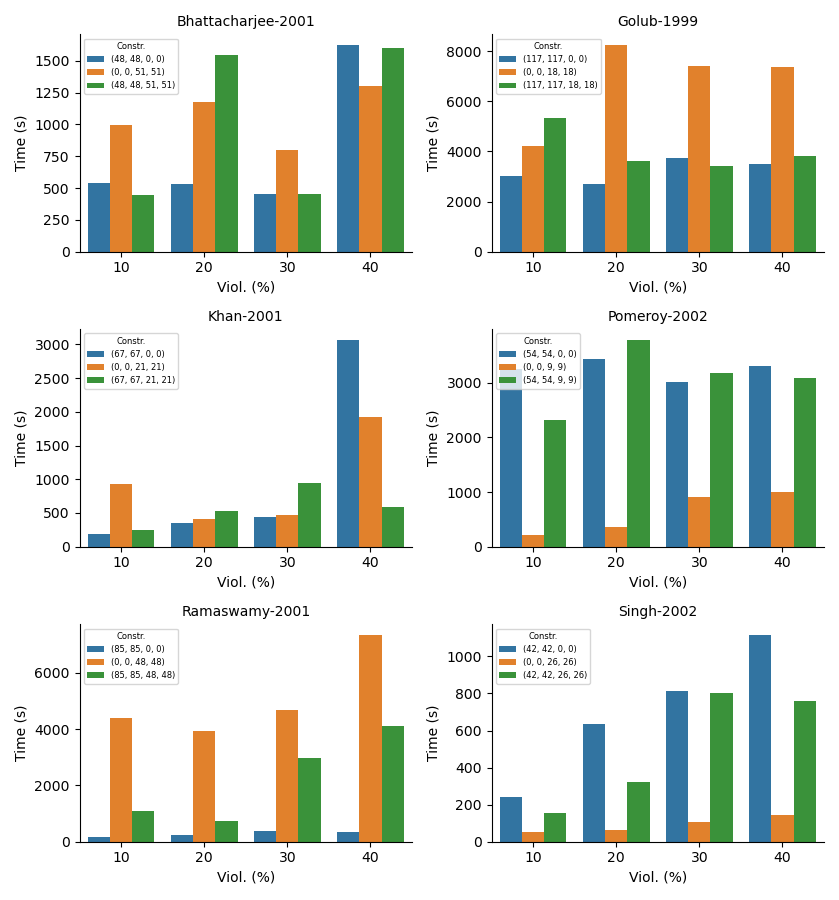}
        \caption{Total solution times of {\tt CBICL-BB} on real-world gene expression datasets under constraint configurations (Constr.) $(n/4, n/4, 0, 0)$, $(0, 0, m/4, m/4)$, $(0, 0, m/4, m/4)$, $(n/4, n/4, m/4, m/4)$ and violation percentages (Viol.). Each bar represents the average over three sets of constraints corresponding to the same configuration.}
        \label{fig:n4}
    \end{figure}
    \begin{figure}[!ht]
        \centering
        \includegraphics[scale=0.62]{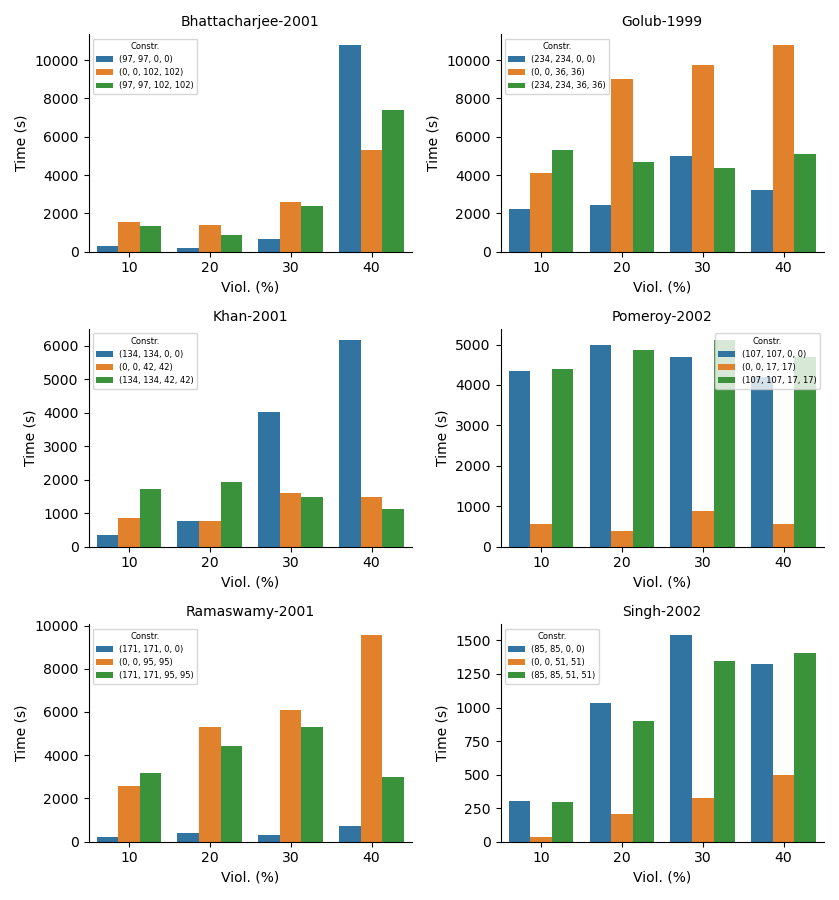}
        \caption{Total solution times of {\tt CBICL-BB} on real-world gene expression datasets under constraint configurations (Constr.) $(n/2, n/2, 0, 0)$, $(0, 0, m/2, m/2)$, $(0, 0, m/2, m/2)$, $(n/2, n/2, m/2, m/2)$ and violation percentages (Viol.). Each bar represents the average over three sets of constraints corresponding to the same configuration.}
        \label{fig:n2}
    \end{figure}
    \clearpage
    \begin{figure}[!ht]
        \centering
        \includegraphics[scale=0.52]{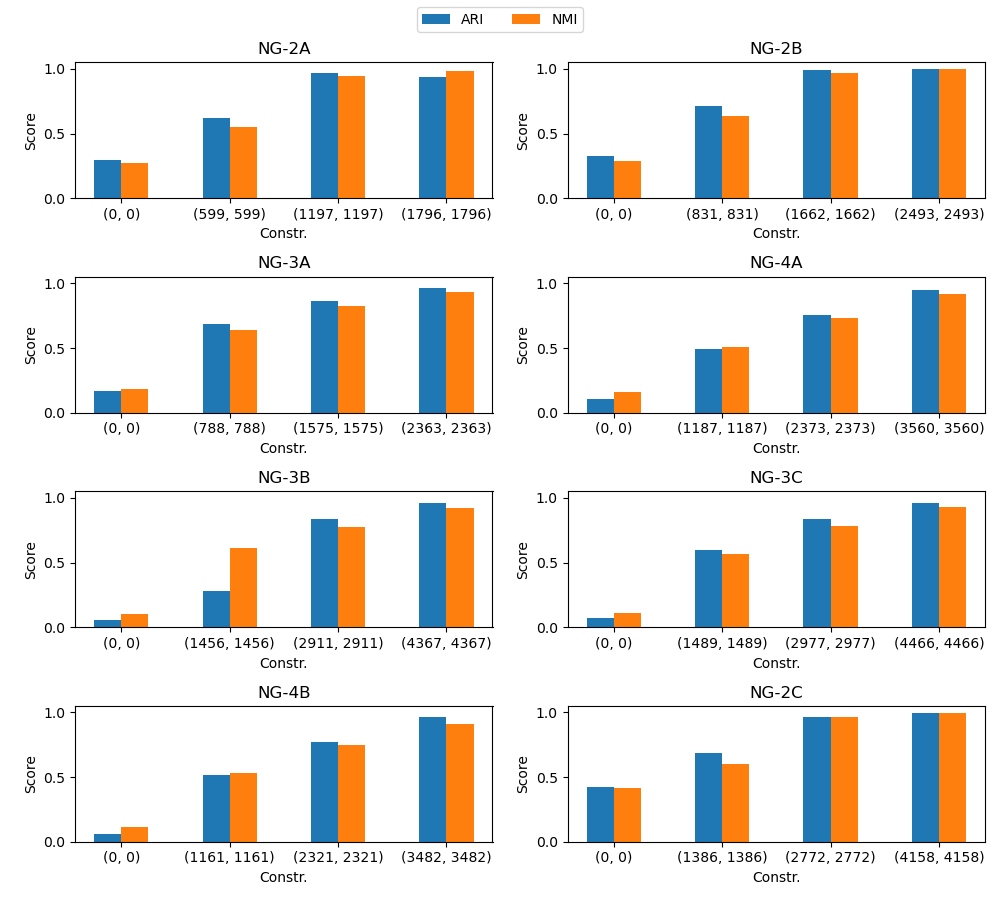}
        \caption{ARI and NMI scores of {\tt CBICL-LR} on real-world document clustering datasets under varying constraint configurations (Constr.). Each bar represents the average over three sets of constraints for each configuration.}
        \label{fig:ari}
    \end{figure}

\end{appendices}

%%===========================================================================================%%
%% If you are submitting to one of the Nature Portfolio journals, using the eJP submission   %%
%% system, please include the references within the manuscript file itself. You may do this  %%
%% by copying the reference list from your .bbl file, paste it into the main manuscript .tex %%
%% file, and delete the associated \verb+\bibliography+ commands.                            %%
%%===========================================================================================%%

%\bibliography{sn-bibliography}% common bib file
%% if required, the content of .bbl file can be included here once bbl is generated
%%\input sn-article.bbl

\bibliography{abbr, sn-bibliography}

\end{document}